\documentclass{amsart}

\usepackage{amsmath,amssymb,amsthm}
\usepackage{pinlabel}

\newtheorem{prop}{Proposition}
\newtheorem{thm}[prop]{Theorem}
\newtheorem{lem}[prop]{Lemma}

\theoremstyle{definition}

\newtheorem*{ack}{Acknowledgements}

\def\co{\colon\thinspace}

\newcommand{\tb}{{\tt tb}}
\newcommand{\lk}{{\tt lk}}
\newcommand{\rot}{{\tt rot}}

\newcommand{\rma}{\mathrm{a}}
\newcommand{\rmb}{\mathrm{b}}
\newcommand{\rmc}{\mathrm{c}}
\newcommand{\rmd}{\mathrm{d}}
\newcommand{\rme}{\mathrm{e}}
\newcommand{\rmi}{\mathrm{i}}

\newcommand{\C}{\mathbb{C}}
\newcommand{\CP}{\mathbb{C}\mathrm{P}}
\newcommand{\N}{\mathbb{N}}
\newcommand{\Q}{\mathbb{Q}}
\newcommand{\R}{\mathbb{R}}
\newcommand{\RP}{\mathbb{R}\mathrm{P}}
\newcommand{\Z}{\mathbb{Z}}

\newcommand{\LL}{\mathbb{L}}

\newcommand{\xist}{\xi_{\mathrm{st}}}

\newcommand{\lambdac}{\lambda_{\mathrm{c}}}

\DeclareMathOperator{\PD}{\mathrm{PD}}

\begin{document}

\author[H.~Geiges]{Hansj\"org Geiges}
\address{Mathematisches Institut, Universit\"at zu K\"oln,
Weyertal 86--90, 50931 K\"oln, Germany}
\email{geiges@math.uni-koeln.de}

\author[S.~Onaran]{Sinem Onaran}
\address{Department of Mathematics, Hacettepe University,
06800 Beytepe-Ankara, Turkey}
\email{sonaran@hacettepe.edu.tr}

\title[Legendrian rational unknots]{Legendrian
rational unknots in lens spaces}

\date{}

\begin{abstract}
We classify Legendrian rational unknots with tight complements in
the lens spaces $L(p,1)$ up to coarse equivalence.
As an example of the general case, this classification
is also worked out for $L(5,2)$. The knots are described
explicitly in a contact surgery diagram of the corresponding
lens space.
\end{abstract}


\maketitle


\section{Introduction}
This paper is concerned with the classification of
Legendrian rational unknots in lens spaces. The lens space in question
may be equipped with a tight or an overtwisted contact structure,
but in the latter case we require that the knot complement be
tight. Legendrian knots in overtwisted contact $3$-manifolds
with tight complement are called \emph{non-loose} or \emph{exceptional}.

The classification of Legendrian unknots in $S^3$ is
due to Eliashberg and Fraser~\cite{elfr09}. In the case
of the tight standard contact structure $\xist$ on $S^3$, their
classification is up to isotopy; in the case of exceptional unknots
in an overtwisted contact structure, up
to coarse equivalence. Recall that
two Legendrian knots $L_i\subset (M_i,\xi_i)$, $i=1,2$, in contact
$3$-manifolds are called \emph{coarsely equivalent} if there
is a contactomorphism $(M_1,\xi_1)\rightarrow (M_2,\xi_2)$
carrying $L_1$ to $L_2$.
Here the contact structures are understood to be (co-)oriented, and the
contactomorphism is supposed to preserve the (co-)orientation.
The classification of Legendrian knots in overtwisted contact manifolds
up to isotopy is complicated by the fact there are contactomorphisms
topologically but not contact isotopic to the identity, cf.\ the
discussion in \cite[Section~4.3]{elfr09}.

In the present paper we extend the classification result of Eliashberg
and Fraser to rational unknots in lens spaces. Our focus will lie on the
exceptional case, and we too are content with the
classification up to coarse equivalence. The classification
in the tight case is essentially due to Baker and
Etnyre~\cite{baet12}, although they give an explicit description
only for $L(p,1)$ with $p$ odd.

We obtain a complete classification (both in the tight
and the exceptional case) for the lens spaces $L(p,1)$
with $p$ any integer. As an illustration of the general case we also
discuss the classification for $L(5,2)$. In particular,
we determine the range of the classical
invariants realisable by such rational unknots, and the
$3$-dimensional homotopy invariant of the contact structures
containing exceptional knots.

This classification is achieved
as follows. The number of distinct Legendrian rational unknots
with tight complements is determined via the classification
of tight contact structures on solid tori; this strategy has
previously been employed by Etnyre~\cite{etny}. We then describe
the expected number of Legendrian rational unknots
explicitly in a contact surgery diagram of the lens space.
For the $3$-sphere, such a description is due to
Plamenevskaya~\cite{plam12}. In all cases the knots are distinguished
by the rational analogues of the classical Legendrian knot
invariants.

Here is an outline of the paper.
In Section~\ref{section:unknots} we recall the topological
classification of rational unknots in lens spaces. In
Section~\ref{section:classical} we describe
a result of Lisca et al.~\cite{loss09} about the computation of
the classical Legendrian knot invariants of a Legendrian
knot presented in a contact surgery diagram. We extend
their result from integral to rational homology spheres and
from nullhomologous to rationally nullhomologous knots. The
invariants in this case are the rational Legendrian knot
invariants~\cite{oztu05,bagr09,baet12}.

In Section~\ref{section:plane} we recall how to compute
the $3$-dimensional homotopy invariant of a contact structure
from a surgery diagram. This will be used in some cases to show that
the contact structure defined by a certain surgery diagram is
overtwisted.

Sections \ref{section:3-sphere} to~\ref{section:L52}
contain the classification for $S^3$, $\RP^3$, $L(p,1)$ and
$L(5,2)$, respectively. The classification for
$S^3$ and $\RP^3$ is of course subsumed by that for $L(p,1)$.
Nonetheless, the separate description of those two
simple cases allows us to include some additional details
and to make the whole classification scheme more
transparent. Many of the necessary computations
are relegated to Section~\ref{section:compute}.

We understand that B\"ulent Tosun has been working on the
classification problem discussed in this paper using a
parallel approach, but staying closer to the argument in
\cite{etny} rather than relying on surgery diagrams for the existence
part of the classification. This may in
fact be advantageous for dealing with the general $L(p,q)$.
\section{Rational unknots in lens spaces}
\label{section:unknots}
The lens space $L(p,q)$ with $p\in\N$ and $1\leq q\leq p-1$
coprime to~$p$ is defined as the quotient space of
$S^3\subset\C^2$ under the $\Z_p$-action generated by
\[ (z_1,z_2)\longmapsto (\rme^{2\pi\rmi/p}z_1,\rme^{2\pi\rmi q/p}z_2).\]
This gives $L(p,q)$ a canonical orientation, and our contact structures
are assumed to be positive for that orientation.
Alternatively, $L(p,q)$ with the described orientation can be obtained
by surgery along a single $(-p/q)$-framed unknot in~$S^3$.

A \emph{rational unknot} $K$ in some $3$-manifold
$M$ is a knot with a rational Seifert disc,
i.e.\ some cable of $K$ on the boundary
$\partial(\nu K)$ of a tubular neighbourhood of~$K$
is supposed to bound a $2$-disc $D$ embedded in $M\setminus\nu K$,
cf.~\cite{baet12}. As discussed in that paper,
the union $\nu K\cup D$ equals the complement of an open ball
in a lens space, so for the study of rational unknots
one may restrict attention to the case of $M$ being a lens space.

Moreover, a rational unknot\footnote{When we speak of a `rational unknot'
in $L(p,q)$ we always mean a rational unknot that is not
an honest unknot, i.e.\ the homological order of the knot is
supposed to be greater than~$1$.}
in $L(p,q)$ is then necessarily
the spine of one of the Heegaard tori. Recall that the
genus~1 Heegaard splitting of a lens space is unique up to
isotopy~\cite{bona83}. Hence, up to isotopy there are at most
four oriented rational unknots in $L(p,q)$, namely $\pm K_j$, $j=1,2$,
where
\[ K_1=\{ [\rme^{\rmi\theta},0]\co 0\leq\theta\leq 2\pi/p\}\subset
L(p,q),\]
and likewise for $K_2$. For $p=2$ this reduces in fact to
one possibility, and for $p>2$, $q\in\{1,p-1\}$ the knot $\pm K_1$
is isotopic to $\pm K_2$, both being fibres in an $S^1$-bundle over $S^2$.
In homology one has $[K_2]=q[K_1]$, so for $q\not\in\{1,p-1\}$
there are indeed four rational unknots up to isotopy,
see~\cite[Lemma~5.2]{baet12}.

In the surgery picture, $K_1$ can be represented
as in Figure~\ref{figure:standard-unknot}. The knot $K_2$ would be the
spine of the solid torus glued in to perform the surgery.
By exchanging the role of the two Heegaard tori,
which induces the orientation-preserving diffeomorphism
$L(p,q)\cong L(p,r)$ for $qr\equiv 1$ mod~$p$,
one gets a similar picture for $K_2$, with the surgery coefficient
replaced by $-p/r$.

\begin{figure}[h]
\labellist
\small\hair 2pt
\pinlabel $K_1$ [r] at 0 18
\pinlabel $-p/q$ [l] at 70 20
\endlabellist
\centering
\includegraphics[scale=1]{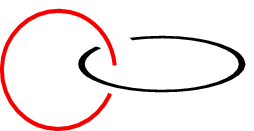}
  \caption{One of the rational unknots in $L(p,q)$.}
  \label{figure:standard-unknot}
\end{figure}

By expanding the rational surgery into integral surgeries
along a link of unknots, one can give representations of
$K_1$ and $K_2$ in a single surgery diagram. For instance,
in the case $p=5$, $q=2$ we can take $r=3$. Since
$-5/2=-3-1/(-2)$ and $-5/3=-2-1/(-3)$ we can represent
$K_1$ and $K_2$ as in Figure~\ref{figure:K1K2}. A slam dunk,
cf.~\cite[Figure~5.30]{gost99}, will then produce
Figure~\ref{figure:standard-unknot} or the analogous
picture for~$K_2$, respectively.

\begin{figure}[h]
\labellist
\small\hair 2pt
\pinlabel $K_1$ [r] at 0 18
\pinlabel $K_2$ [l] at 175 18
\pinlabel $-3$ [tl] at 38 12
\pinlabel $-2$ [tl] at 63 12
\pinlabel $-3$ [tr] at 109 12
\pinlabel $-2$ [tr] at 135 12
\endlabellist
\centering
\includegraphics[scale=1]{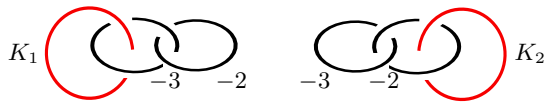}
  \caption{The two rational unknots in $L(5,2)$.}
  \label{figure:K1K2}
\end{figure}

Since we are interested in a classification up to coarse equivalence,
we need to take the action of the diffeomorphism group into account.
In this topological setting, coarse equivalence of two knots
is supposed to mean that there is an orientation-preserving
diffeomorphism of the ambient manifold sending one knot to
the other.

\begin{prop}
Up to coarse equivalence there is exactly one oriented rational unknot
in $L(p,q)$ for $q^2\equiv 1$ {\rm mod}~$p$, and exactly two
for $q^2\not\equiv 1$ {\rm mod}~$p$.
\end{prop}

\begin{proof}
The orientation-preserving diffeomorphism of $L(p,q)$
induced by $(z_1,z_2)\mapsto (\overline{z}_1,\overline{z}_2)$
sends $K_j$ to $-K_j$, $j=1,2$. Thus, topologically we can ignore
the orientation of $K_j$.

For $q^2\equiv 1$ {\rm mod}~$p$, the orientation-preserving diffeomorphism
of $L(p,q)$ induced by $(z_1,z_2)\mapsto (z_2,z_1)$ exchanges
$K_1$ and $K_2$. For $q^2\not\equiv 1$ {\rm mod}~$p$,
there are only two orientation-preserving diffeomorphisms 
of $L(p,q)$ up to isotopy: the identity and the one
described above, see~\cite{mccu02}. So $K_1$ and $K_2$ cannot be
coarsely equivalent.
\end{proof}

However, as pointed out in~\cite{baet12} for the tight case,
the two different orientations of $K_1$ may correspond to
non-equivalent Legendrian realisations. Similar
considerations apply in the exceptional case. This gives some
information about the contactomorphism group of
the corresponding contact structure.
\section{The classical invariants}
\label{section:classical}
In this section we first recall from~\cite{baet12} how to define
the classical invariants
for \emph{rationally} nullhomologous Legendrian knots. We then
show how to compute these invariants for knots presented in
a surgery diagram of the ambient manifold, provided this manifold
is a rational homology sphere. This constitutes a mild extension of a result
due to Lisca et al.~\cite[Lemma~6.6]{loss09}.
\subsection{Definition of the invariants}
Let $L\subset (Y,\xi)$ be a rationally nullhomologous
Legendrian knot in a contact $3$-manifold, i.e.\ $L$ is
of some order $r\in\N$ in $H_1(Y)$. Then there is a rational Seifert
surface $\Sigma\subset Y$ for $rL$, i.e.\ a surface that is embedded,
except along its boundary, which is an $r$-fold covering of~$L$.
Note that the boundary of $\Sigma$ need not be connected. (One can
replace $L$ by a suitable embedded curve  or collection of curves
on the boundary of a tubular neighbourhood of~$L$, representing
the class $rL$ in the tubular neighbourhood, and then find
an embedded surface with that curve (or those curves) as its boundary.)
Let $L'$ be a push-off of $L$ in the direction of the contact framing,
i.e.\ the framing determined by $\xi|_L$. Then the
\emph{rational Thurston--Bennequin invariant} of $L$ is
\[ \tb_{\Q}(L):=\frac{1}{r}L'\bullet\Sigma,\]
i.e.\ the rational linking number of $L$ and $L'$.
Any two rational Seifert surfaces for $L$ differ by a class in
$H_2(Y)$. Since $L'$ is rationally nullhomologous, its intersection
number with such a class is zero, so $\tb_{\Q}(L)$ is well defined.

Now assume that $L$ is oriented. The contact structure $\xi$
is a trivial plane bundle when restricted to the rational Seifert
surface~$\Sigma$, and the \emph{rational rotation number}
$\rot_{\Q}(L)$ is defined by writing $r\cdot\rot_{\Q}(L)$ for the number
of full turns of the positive tangent vector to $L$, as $L$ is traversed
$r$ times, relative to the trivialisation of $\xi|_{\Sigma}$.
In general, this number will depend on the relative homology
class represented by~$\Sigma$. If the Euler class $e(\xi)$ is a torsion
class, then $\rot_{\Q}(L)$ is well defined.
\subsection{Computation in a surgery diagram}
As shown in~\cite{dige04}, any (closed, connected) contact $3$-manifold
$(Y,\xi)$ has a contact $(\pm 1)$-surgery presentation
$\LL=\LL_+\sqcup\LL_-\subset(S^3,\xist)$, i.e.\ there is a Legendrian
link $\LL$ in $S^3$ with its standard tight contact structure such 
that contact $(\pm 1)$-surgery along the components
of $\LL_{\pm}$ produces $(Y,\xi)$.

Now let $\LL=\LL_+\sqcup\LL_-$ be a contact
$(\pm 1)$-surgery presentation of a rational homology $3$-sphere
$(Y,\xi)$. Then $H^2(Y)$ will be a finite abelian group,
so the Euler class $e(\xi)$ is a torsion class.
Furthermore, let $L_0\subset(S^3\setminus\LL,\xist)$
be a Legendrian knot that becomes rationally nullhomologous
in $(Y,\xi)$. Denote the link components of $\LL$ by
$L_1,\ldots ,L_n$, and set $a_i=\tb(L_i)\pm 1$, depending on
whether $L_i$ belongs to $\LL_+$ or~$\LL_-$. So
$a_i$ is the integral surgery coefficient of the link component~$L_i$.
Write $M$ for the linking matrix of~$\LL$, i.e.\
\[ M:=(m_{ij})_{i,j=1}^n,\;\;\text{where}\;\;
m_{ij}:=\begin{cases}
a_i          & \text{if $i=j$,}\\
\lk(L_i,L_j) & \text{if $i\neq j$}.
\end{cases} \]
Define an extended matrix by
\[ M_0:=(m_{ij})_{i,j=0}^n,\;\;\text{where}\;\;
m_{ij}:=\begin{cases}
0            & \text{if $i=j=0$,}\\
a_i          & \text{if $i=j\neq 0$,}\\
\lk(L_i,L_j) & \text{if $i\neq j$}.
\end{cases}\]
In other words, $M_0$ is the linking matrix of $L_0\sqcup\LL$,
with the convention that $\lk(L_0,L_0)$ is set to~$0$.
As a final piece of notation, we write
$\rot_i$ for the rotation number of $L_i$, $i=0,\ldots,n$,
and $\tb_0$ for the Thurston--Bennequin invariant of~$L_0$,
all regarded as knots in $(S^3,\xist)$.

We can now formulate a lemma that tells us how to compute the
classical invariants of $L_0$ when it is regarded as a Legendrian knot
in the surgered manifold $(Y,\xi)$. For the case that $(Y,\xi)$
is an integral homology sphere, this lemma is
due to Lisca et al.~\cite[Lemma~6.6]{loss09}. (The condition
that $Y$ be an integral homology sphere is not contained
in the statement of their lemma, but in the paragraphs preceding it,
and it is used implicitly in their argument.)

\begin{lem}
\label{lem:loss}
The rational invariants of $L_0\subset(Y,\xi)$ are given by
\[ \rot_{\Q}(L_0)=\rot_0-\left\langle
\left(\begin{array}{c}\rot_1\\ \vdots\\ \rot_n\end{array}\right),
M^{-1}
\left(\begin{array}{c}\lk(L_0,L_1)\\ \vdots\\ \lk(L_0,L_n)\end{array}\right)
\right\rangle\]
and
\[ \tb_{\Q}(L_0)=\tb_0+
\frac{\det M_0}{\det M}. \]
\end{lem}

These formulae are exactly the same as those in~\cite{loss09}, with $\tb$
and $\rot$ replaced by their rational counterparts. However, since it
is not entirely clear where the order of $L$ in $H_1(Y)$ might or
might not be relevant for the argument, we deem it worth to
include a proof.
\subsection{The relative Euler class}
Before we turn to that proof, we want to discuss
the behaviour of the relative Euler class of a contact structure
under a contact $(\pm 1)$-surgery along a Legendrian
knot~$L_i$. A neighbourhood of $L_i$ can be identified with
a neighbourhood of $S^1\times\{0\}\subset S^1\times\R^2$,
equipped with the contact structure $\ker(\cos\theta\, dx-\sin\theta\, dy)$.
We think of $S^1\times D^2\subset S^1\times\R^2$
as the solid torus we cut out during the
surgery. The boundary $S^1\times \partial D^2$ of this solid torus is
a convex surface with two dividing curves
$(\pm\sin\theta,\pm\cos\theta,\theta)$. These curves lie in the
class of the longitude $\lambdac$ giving the contact framing.
Write $\mu$ for the meridian of $S^1\times\partial D^2$ and
$\lambda$ for the standard longitude $S^1\times\{*\}$.
Then $\lambdac=\lambda-\mu$. Thus,
in terms of the standard meridian and longitude $\mu,\lambda$,
the slope of the dividing curves is~$-1$.
The convex torus $S^1\times \partial D^2$
has a linear Legendrian ruling given by the $\theta$-curves, which
represent the class $\lambda=\mu+\lambdac$.

Write $\mu',\lambda'$ for meridian and longitude, respectively, of a
solid torus we glue in to perform the surgery. Contact $(-1)$-surgery can be
described by the gluing maps
\[ \mu'\longmapsto\mu-\lambdac,\;\;\;\lambda'\longmapsto\mu;\]
contact $(+1)$-surgery, by the maps
\[ \mu'\longmapsto\mu+\lambdac,\;\;\;\lambda'\longmapsto\mu+2\lambdac.\]
Note that in both cases $\lambda'-\mu'$ gets glued to $\lambdac$.
So the slope of the dividing curves on the boundary of
the solid torus we want to glue in is again~$-1$, now with respect
to $\mu',\lambda'$.

The Legendrian ruling of the convex torus $S^1\times\partial D^2$
in the original model can be changed from the class
$\lambda=\lambdac+\mu$ to either $\mu$ or $\mu+2\lambdac$ by
an isotopic deformation of the $2$-torus through convex tori
of slope~$-1$ and linear Legendrian ruling,
staying inside any arbitrarily small neighbourhood
of the original torus. After this modification, we see that
contact $(\pm 1)$-surgery simply corresponds to regluing
a standard solid torus with slope $-1$ and linear Legendrian
ruling in the class~$\lambda_0$.

Now suppose that we perform contact $(\pm 1)$-surgery
along a Legendrian knot $L$ in a contact $3$-manifold $(M,\xi_0)$.
The relative Euler class $e(\xi_0,L)\in H^2(M,L)$
is Poincar\'e dual to the class in $H_1(M\setminus L)$ represented
by the zero set of a generic section of~$\xi_0$ that coincides with
the tangent direction along~$L$. By what we just discussed, we may assume
that this section coincides with the Legendrian ruling
on the boundary of a standard tubular neighbourhood of~$L$,
and this section will extend without zeros over the solid torus we
glue in when performing a surgery along~$L$. Translated into our
situation at hand, this implies the following statement.

\begin{lem}
\label{lem:Euler}
Under the natural map
\[ H_1(S^3\setminus (L_0\sqcup\ldots\sqcup L_n))\longrightarrow
H_1(Y\setminus L_0)\]
induced by inclusion, the Poincar\'e dual of the relative Euler
class $e(\xist,L_0\sqcup\ldots\sqcup L_n)$ maps to the Poincar\'e dual
of $e(\xi,L_0)$.\hfill\qed
\end{lem}
\subsection{Proof of Lemma~\ref{lem:loss}}
Write $\mu_i$ for the meridian of $L_i$, and $\lambda_i$
for the longitude determined by the surface framing. Then
\[ H_1(S^3\setminus\bigsqcup_{i=0}^n L_i)\cong\Z_{\mu_0}\oplus\cdots\oplus
\Z_{\mu_n},\]
where $\Z_{\mu}$ denotes a copy of the integers generated by the
class~$\mu$. In $S^3\setminus\bigsqcup_{i=0}^n L_i$, the
longitude $\lambda_i$ represents the class
\begin{equation}
\label{eqn:lambda}
\lambda_i=\sum_{\substack{j=0\\j\neq i}}^n\lk(L_i,L_j)\mu_j.
\end{equation}
Surgery with coefficient $a_i$ along $L_i$ means that we
glue a new meridional disc along $a_i\mu_i+\lambda_i$, $i=1,\ldots ,n$.
It follows that
\[ H_1(Y\setminus L_0)\cong\Z_{\mu_0}\oplus\cdots\oplus\Z_{\mu_{n}}\slash
\langle a_i\mu_i+\sum_{\substack{j=0\\j\neq i}}^n\lk(L_i,L_j)\mu_j=0,\,
i=1,\ldots ,n\rangle.\]

On the other hand, from the Mayer--Vietoris sequence of the
triple $(Y; Y\setminus L_0,L_0)$ and the assumption that $Y$
be a rational homology sphere (as well as some
obvious identifications under excision isomorphisms)
we have the short exact sequence
\[ 0\longrightarrow H_1(T^2)\longrightarrow
H_1(Y\setminus L_0)\oplus H_1(L_0)\longrightarrow H_1(Y)\longrightarrow 0,\]
with $H_1(Y)$ a finite abelian group. We have $H_1(T^2)\cong\Z_{\mu_0}\oplus
\Z_{\lambda_0}$. The class $\mu_0$ maps to $0$ in $H_1(L_0)\cong\Z$;
the class $\lambda_0$, to~$1$. So the sequence reduces to
\begin{equation}
\label{eqn:sequence}
0\longrightarrow \Z_{\mu_0}\longrightarrow H_1(Y\setminus L_0)
\longrightarrow H_1(Y)\longrightarrow 0.
\end{equation}
(Alternatively, this follows from $L_0$ being rationally nullhomologous
in~$Y$.)

Hence the Poincar\'e dual of the relative Euler class $e(\xi,L_0)$
over the rationals,
\[ \PD(e(\xi,L_0))_{\Q}:=\PD(e(\xi,L_0))\otimes_{\Z}1\in H_1(Y\setminus L_0;
\Q)\cong\Q_{\mu_0}, \]
is some rational multiple of~$\mu_0$.
Beware that --- over the integers ---
$\mu_0$ is not, in general, a primitive element in $H_1(Y\setminus L_0)$.
For instance, for $Y=\RP^3$ and $L_0$ representing the generator
of $\pi_1(\RP^3)=\Z_2$, the class $\mu_0$ is twice the generator
of $H_1(\RP^3\setminus L_0)\cong\Z$.

The definition of the rotation number of a Legendrian knot can
be interpreted in terms of relative Euler classes. This translates into
\[ \PD\bigl(e(\xist,\bigsqcup_{i=0}^nL_i)\bigr)=\sum_{i=0}^n\rot_i\mu_i.\]
For the rational rotation number $\rot_{\Q}(L_0)$ of $L_0\subset(Y,\xi)$
we argue similarly.
If the order of $L_0$ in $H_1(Y)$ is~$r$, and $\Sigma$ is a rational
Seifert surface for $L_0$ in~$Y$, then
\[ r\cdot \rot_{\Q}(L_0)=\langle e(\xi,L_0),[\Sigma]\rangle=
\PD(e(\xi,L_0))\bullet\Sigma.\]
For this intersection product, only the free part of $\PD(e(\xi,L_0))$
is relevant, and since the intersection product $\mu_0\bullet[\Sigma]$
equals~$r$, we conclude that
\[ \PD(e(\xi,L_0))_{\Q}=\rot_{\Q}(L_0)\mu_0.\]
Hence, by Lemma~\ref{lem:Euler},
\[ \sum_{i=0}^n\rot_i\mu_i=\rot_{\Q}(L_0)\mu_0\;\;
\text{in $H_1(Y\setminus L_0;\Q)$}.\]
The relations in the presentation of $H_1(Y\setminus L_0)$
can be written formally as
\[ M\left(\begin{array}{c}\mu_1\\ \vdots\\ \mu_n\end{array}\right)
+\left(\begin{array}{c}\lk(L_0,L_1)\\ \vdots\\ \lk(L_0,L_n)\end{array}\right)
\mu_0=0.\]
The surgery description of $Y$ defines a $4$-dimensional handlebody $X$
with boundary $\partial X=Y$. Both $H_2(X)$ and $H_2(X,Y)$
are isomorphic to $\Z^n$. The relevant part of the
homology exact sequence of the pair $(X,Y)$ is
\[ H_2(Y)\longrightarrow H_2(X)\stackrel{M}{\longrightarrow} H_2(X,Y)
\longrightarrow H_1(Y).\]
Since $Y$ is a rational homology sphere, we have $H_2(Y)=0$,
so the matrix $M$ is invertible over~$\Q$. Therefore, 
in $H_1(Y\setminus L_0;\Q)$ we have
\[ \rot_{\Q}(L_0)\mu_0=\sum_{i=0}^n\rot_i\mu_i= \Bigl(\rot_0- \left\langle
\left(\begin{array}{c}\rot_1\\ \vdots\\ \rot_n\end{array}\right),
M^{-1}
\left(\begin{array}{c}\lk(L_0,L_1)\\ \vdots\\ \lk(L_0,L_n)\end{array}\right)
\right\rangle
\Bigr)\mu_0.\]
This proves the first formula in the lemma.

We now turn to the Thurston--Bennequin invariant. The contact framing
of $L_0$ (both in $(S^3,\xist)$ and in $(Y,\xi)$) is given by
$\tb_0\mu_0+\lambda_0$. On the other hand, from the exact
sequence~(\ref{eqn:sequence}) we see that there is a unique $a_0\in\Z$
such that $a_0\mu_0+r\lambda_0$ is nullhomologous in $Y\setminus L_0$,
where as before $r$ denotes the order of $L_0$ in~$Y$. Then the
rational Thurston--Bennequin invariant of $L_0$ in $(Y,\xi)$ can be
computed as
\begin{equation}
\label{eqn:tb}
r\cdot\tb_{\Q}(L_0)=(\tb_0\mu_0+\lambda_0)\bullet(a_0\mu_0+r\lambda_0)=
r\cdot\tb_0-a_0,
\end{equation}
where the intersection product should be interpreted as a product
on the boundary of a tubular neighbourhood of~$L_0$.

From~(\ref{eqn:lambda}) we have
\[ a_0\mu_0+r\lambda_0=a_0\mu_0+r\sum_{j=1}^n\lk(L_0,L_j)\mu_j,\]
and the fact that this is nullhomologous in $Y\setminus L_0$
means that it can be expressed as a linear combination of the
relations in $H_1(Y\setminus L_0)$, which yields
\[ 0=\left|\begin{array}{cccc}
a_0          & r\cdot\lk(L_0,L_1) & \cdots & r\cdot\lk(L_0,L_n)\\
\lk(L_1,L_0) & a_1                & \cdots & \lk(L_1,L_n)\\
\vdots       & \vdots             & \ddots & \vdots \\
\lk(L_n,L_0) & \lk (L_n,L_1)      & \cdots & a_n
\end{array}\right|
=a_0\det M+r\det M_0.\]
With (\ref{eqn:tb}) we get
\[ \tb_{\Q}(L)=\tb_0-\frac{a_0}{r}=\tb_0+\frac{\det M_0}{\det M}.\]
\section{Invariants of tangent $2$-plane fields}
\label{section:plane}
A surgery presentation $\LL=\LL_+\sqcup\LL_-\subset(S^3,\xist)$
of a given contact $3$-manifold $(Y,\xi)$ determines
a $4$-dimensional $2$-handlebody $X$ with boundary $\partial X=Y$.
Following \cite{gomp98} and \cite{dgs04} we are now going to explain
how to determine the homotopical data
of $\xi$ as a tangent $2$-plane field from such a presentation.

Given a choice of spin structure $s$ on $Y$, there is an
invariant $\Gamma(\xi,s)\in H_1(Y)$ of the homotopy type
of $\xi$ over the $2$-skeleton of~$Y$.
When the first Chern class $c_1(\xi)$ is a torsion class,
the homotopy obstruction over the $3$-skeleton can be
described by a rational number $d_3(\xi)$.

In \cite[Theorem~4.12]{gomp98} it is shown how to compute
$\Gamma(\xi,s)$ from a surgery diagram containing only
contact $(-1)$-surgeries. A spin structure $s$ can be
specified in terms of a characteristic sublink in the surgery
diagram. B\"ulent Tosun has informed us of a way to compute
the $2$-dimensional homotopy invariant $\Gamma(\xi,s)$
from any contact surgery diagram, including those
with $1$-handles and contact $(+1)$-surgeries.
This would allow one to give a complete homotopy classification of
the contact structures on $L(p,q)$ we describe in terms
of surgery diagrams in the following sections.

The $3$-dimensional invariant can be computed as follows.
Write $\sigma(X)$ for the signature of~$X$,
and $\chi (X)$ for its Euler characteristic. Let 
$\Sigma_i\subset X$ be the surface obtained by gluing
a Seifert surface of $L_i$ with the core disc
of the handle corresponding to $L_i\subset\LL$.
The homology class of $\Sigma_i$ in $H_2(X)$ is
completely determined by $L_i\subset S^3$.
Generalising a result of Gompf, the following was shown
in~\cite{dgs04}.

\begin{prop}
\label{prop:d3}
Suppose that $c_1(\xi)$ is torsion, and $\tb (L_i)\neq 0$
for each $L_i\in\LL_+$. Then
\begin{equation}
\label{eqn:d3}
d_3(\xi )=\frac{1}{4}\bigl( c^2-3\sigma (X)-2\chi (X)\bigr) +q,
\end{equation}
where $q$ denotes the number of components of $\LL_+$,
and $c\in H^2 (X)$ is the cohomology class
determined by $c(\Sigma_i)=\rot(L_i)$ for each $L_i\subset \LL$.
\end{prop}

See \cite{dgs04} of an extensive discussion of this formula,
in particular concerning the computation of the term $c^2$.

The standard (and unique tight) contact structure $\xist$ on $S^3$ has
$d_3(\xist)=-1/2$. On $L(p,1)$ there are, by~\cite{giro00} and~\cite{hond00},
exactly $p-1$ tight contact structures up to isotopy.
They can be obtained by contact $(-1)$-surgery on $S^3$
along a single unknot with invariants $\tb=-p+1$ and
\[ \rot\in\{ -p+2,-p+4,\ldots ,p-4,p-2\},\]
obtained from the standard Legendrian unknot ($\tb=-1,\rot=0$)
by any mix of $p-2$ positive or negative stabilisations.
The $d_3$-invariant of the corresponding contact structure on
$L(p,1)$ is given by
\[ d_3=-\frac{1}{4}\bigl(1+\frac{\rot^2}{p}\bigr).\]
\section{The $3$-sphere}
\label{section:3-sphere}
Topologically trivial Legendrian knots in arbitrary
tight contact $3$-manifolds were shown
by Eliashberg and Fraser~\cite{elfr09} to be classified up
to Legendrian isotopy by the classical invariants $\tb$ and $\rot$,
and they determined the range of these invariants.
So part (a) of the following theorem is a weaker formulation
of their result, which we include for completeness and comparison
with the case of exceptional knots.

The exceptional unknots in the $3$-sphere $S^3$ have also
been classified, up to coarse equivalence,
by Eliashberg and Fraser~\cite[Theorem~4.7]{elfr09}.
An alternative proof of their result was given by Etnyre and
Vogel, see~\cite{etny}. We are going to give yet another
proof of this classification, which contains \emph{in nuce}
all the key ideas required to extend the result to lens spaces.
Our argument for determining an upper bound on the number
of exceptional knots is parallel to that of~\cite{etny}.
The proof is then completed by finding as many explicit
realisations of exceptional unknots as this bound allows.
In the case of $S^3$, these explicit realisations are due
to Plamenevskaya~\cite{plam12}.

\begin{thm}[Eliashberg--Fraser]
\label{thm:elfr}
{\rm (a)} Let $L\subset (S^3,\xist)$ be a Legendrian unknot. Then
$\tb(L)=n$ with $n$ a negative integer, and $\rot(L)$ lies in the range
\[ \{ n+1, n+3,\ldots ,-n-3,-n-1\}.\]
Any such pair of invariants $(\tb,\rot)$ is realised, and it
determines $L$ up to coarse equivalence,
i.e.\ for each $n\leq -1$ we have $|n|$ distinct Legendrian unknots.

{\rm (b)} Let $L\subset(S^3,\xi)$ be an exceptional unknot in
an overtwisted contact structure $\xi$ on~$S^3$. Then
$\xi$ is the contact structure determined up to
isotopy by $d_3(\xi)=1/2$, and
\[ \bigl(\tb(L),\rot(L)\bigr)\in\bigl\{(n,\pm(n-1))\co n\in\N\bigr\}.\]
These invariants determine $L$ up to coarse equivalence, and any pair
of invariants in this set is realised.
\end{thm}

\begin{proof}
Examples of Legendrian unknots in $(S^3,\xist)$ that
realise the invariants stated in the theorem are given
by arbitrary stabilisations of a standard Legendrian
unknot with $\tb=-1$ and $\rot=0$.

Examples of exceptional unknots with
the stated invariants have been described by Plamenevskaya~\cite{plam12}
in terms of the front projection of the knot in a
contact surgery diagram,
see Figure~\ref{figure:olga-lens}.

\begin{figure}[h]
\labellist
\small\hair 2pt
\pinlabel $(\rma)$ at 0 167
\pinlabel $(\rmb)$ at 0 55
\pinlabel $(\rmc)$ at 88 167
\pinlabel $L$ [bl] at 66 132
\pinlabel $+1$ [l] at 66 148
\pinlabel $+1$ [l] at 66 107
\pinlabel $L$ [bl] at 64 39
\pinlabel $+1$ [l] at 64 14
\pinlabel $L$ [bl] at 146 39
\pinlabel $+1$ [l] at 146 14
\pinlabel $-1$ [l] at 146 58
\pinlabel $-1$ [l] at 146 76
\pinlabel $-1$ [l] at 146 132
\pinlabel $-1$ [l] at 146 150
\endlabellist
\centering
\includegraphics[scale=1]{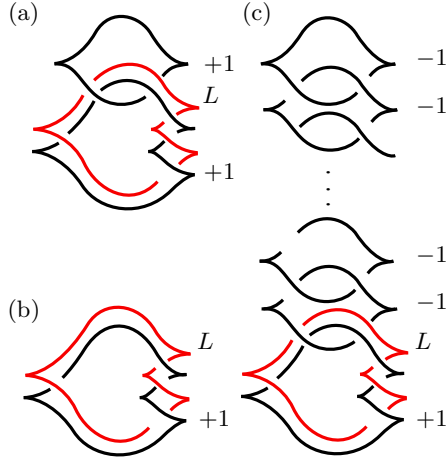}
  \caption{The exceptional unknots in the $3$-sphere.}
  \label{figure:olga-lens}
\end{figure}

Each of these diagrams gives a copy of $S^3$, as can be seen
by simple Kirby moves, cf.~\cite{plam12}.
A straightforward computation with the formula
from Proposition~\ref{prop:d3} shows that each
diagram gives a contact structure on $S^3$ with $d_3=1/2$.
So this contact structure is overtwisted (and determined
up to isotopy by this value of~$d_3$ thanks to Eliashberg's
classification of overtwisted contact structures~\cite{elia89},
cf.~\cite{geig08}). Contact $(-1)$-surgery along $L$
cancels the $(+1)$-surgery along the parallel knot,
see~\cite[Section~3]{dige04} or \cite[Prop.~6.4.5]{geig08}.
This leaves us with a diagram containing only contact
$(-1)$-surgeries, or one with a single $(+1)$-surgery along the
standard Legendrian unknot. The latter produces the
tight contact structure on $S^1\times S^2$, see~\cite[Lemma~4.3]{dgs04},
the former a Stein fillable and hence tight contact structure.
This shows that in all examples the knot $L$ is exceptional.

We claim that the knot $L$ in Figure~\ref{figure:olga-lens}(a)
has $(\tb,\rot) =(1,0)$, the one in (b) has $(\tb,\rot)=(2,\pm 1)$
(depending on a choice of orientation of~$L$), and the one
in (c) (with $n-2\geq 1$ unknots along which we perform
$(-1)$-surgery) has $(\tb,\rot)=(n,\pm(n-1))$.

Plamenevskaya determines $\tb(L)$
by keeping track of the contact framing through Kirby moves;
no comment is made about $\rot(L)$. In fact, the claimed values of
$\tb(L)$ and $\rot(L)$ follow easily from Lemma~\ref{lem:loss}.
See Section~\ref{section:compute} for some details of these
computations.

So we are left with showing that these invariants
(in the tight and exceptional case, respectively) determine
$L$ up to coarse equivalence, and that no other values of
the classical invariants can be realised.

Given a Legendrian unknot $L$ in $(S^3,\xist)$
or an exceptional unknot $L$ in $S^3$, decompose
the $3$-sphere along a torus as $S^3=V_1\cup V_2$, with
$V_1$ a standard neighbourhood of~$L$. More precisely,
with $\mu_i,\lambda_i$ denoting meridian and longitude
of the solid tori $V_i$, we assume that the gluing is described
by the identifications $\mu_1=\lambda_2$, $\lambda_1=\mu_2$,
and $\partial V_1$ is a convex torus with two dividing curves
of slope $1/n$, where $n:=\tb(L)$.

Both in the tight and the exceptional case,
the contact structure on $V_2$ is tight,
and the boundary $\partial V_2$ is convex with two
dividing curves of slope~$n$. Moreover, up to coarse
equivalence $L$ is determined by the contact structure
on~$V_2$.
According to Giroux~\cite{giro00} and
Honda~\cite{hond00}, the number of tight contact structures on
a solid torus $V$ inducing a fixed characteristic foliation on
$\partial V$ divided by two curves of slope $-p/q< -1$
is given by
\[ |(r_0+1)\cdot\ldots\cdot(r_{k-1}+1)\cdot r_k|,\]
where the $r_i<-1$ are the terms in the continued fraction expansion
\[ -\frac{p}{q}=r_0-\cfrac{1}{r_1-\cfrac{1}{r_2-\cdots-\cfrac{1}{r_k}}}
=:[r_0,\ldots,r_k]; \] 
for slope $-1$ there is a unique structure.

For $n<0$ the continued fraction expansion is given by $k=0$ and
$r_0=n$, i.e.\ we have $|n|$ distinct tight structures,
which corresponds to the $|n|$ realisations of a Legendrian
unknot with $\tb=n$ in $(S^3,\xist)$ described above.

For $n=0$, the contact structure on $V_2$ would have to be overtwisted,
so this case does not occur in $(S^3,\xist)$ or when $L$ is exceptional.

Finally, for $n>0$ we first have to modify $V_2$ by a Dehn twist
such that the dividing curves have a slope $\leq -1$, in order
to apply the classification result cited above. A Dehn twist
of $V_2$ that replaces $\lambda_2$ by $\lambda_2'=\lambda_2+k\mu_2$
changes the slope to $s_2'=n/(1-kn)$, since
\[ \mu_2+n\lambda_2=(1-kn)\mu_2+n\lambda_2'.\]
For $n=1$ we have to take $k=2$, which gives $s_2'=-1$. For this slope there
is exactly one tight contact structure on~$V_2$. For $n\geq 2$, we have
to take $k=1$, resulting in a slope $s_2'=-n/(n-1)$. In this case there
are exactly two tight contact structures on~$V_2$, for inductively one
sees that the continued fraction expansion of $-n/(n-1)$ is
$[-2,\ldots,-2]$.

Thus, for $n\geq 1$ the number of tight contact structures on
$V_2$ equals the number of examples in Figure~\ref{figure:olga-lens}.
It follows that these examples constitute a complete list
of exceptional unknots.
\end{proof}
\section{Projective space}
\label{section:rp3}
The following is the analogue of Theorem~\ref{thm:elfr}
for the real projective space $\RP^3=L(2,1)$.

\begin{thm}
\label{thm:rp3}
{\rm (a)} Let $L\subset (\RP^3,\xist)$ be a Legendrian rational unknot
in the unique tight contact structure on $\RP^3$. Then $\tb_{\Q}(L)=n+1/2$
with $n$ a negative integer, and $\rot_{\Q}(L)$ lies in the range
\[ \{ n+1,n+3,\ldots ,-n-3,-n-1 \}.\]
Any such pair of invariants $(\tb_{\Q},\rot_{\Q})$ is realised, and it
determines $L$ up to coarse equivalence,
i.e.\ for each $n\leq -1$ we have $|n|$ distinct Legendrian
rational unknots.

{\rm (b)} Up to coarse equivalence, the exceptional rational unknots $L$
in an overtwisted $(\RP^3,\xi)$ are in one-to-one correspondence
with the following set of values of the classical invariants:
\[ \bigl(\tb_{\Q}(L),\rot_{\Q}(L)\bigr)\in
\bigl\{ (n+1/2,\pm n),\,
(m+1/2,\pm (m-1))\co n\in\N_0,\, m\in\N
\bigr\} .\]
The overtwisted contact structure containing the knots
of the first series has $d_3$-invariant equal to $1/4$,
the one containing the second series has $d_3=3/4$.
\end{thm}

In other words, there is exactly one exceptional rational unknot
with $\tb=1/2$, there are three with $\tb=3/2$, and there are four
each for $\tb=(2n+1)/2$ with $n\geq 2$.

\begin{proof}
Legendrian rational unknots in $(\RP^3,\xist)$ that realise the stated
values of the invariants are given as follows. Represent
$(\RP_3,\xist)$ by $(-1)$-surgery
along a single standard Legendrian unknot in $S^3$ with
$\tb =-1$ and $\rot=0$, and let $L_0$ be a push-off of the surgery curve
with $k-1$ positive and $|n|-k$ negative stabilisations,
$k=1,\ldots,|n|$. Observe that by Lemma~\ref{lem:loss} the
push-off without any stabilisations has $\tb_{\Q}=-1+\frac{-1}{-2}=-1/2$.

Examples of exceptional rational unknots
with the stated invariants are shown in Figures
\ref{figure:rp3-tb3} and~\ref{figure:rp3-tb5}. With some simple
Kirby moves one sees that in all cases $L$ is an isotopic copy
of the standard rational unknot $L_0\subset\RP^3$. We illustrate
this in Section~\ref{section:compute} for the example
in Figure~\ref{figure:rp3-tb5}(a); there we also explain
how $\tb_{\Q}(L)$ can be computed from such Kirby moves instead
of Lemma~\ref{lem:loss}.

\begin{figure}[h]
\labellist
\small\hair 2pt
\pinlabel $(\rma)$ at 0 97
\pinlabel $(\rmb)$ at 103 97
\pinlabel $(\rmc)$ at 203 97
\pinlabel $L$ [bl] at 76 56
\pinlabel $+1$ [l] at 76 44
\pinlabel $+1$ [l] at 76 33
\pinlabel $+1$ [l] at 76 24
\pinlabel $L$ [bl] at 177 56
\pinlabel $+1$ [l] at 177 21
\pinlabel $L$ [bl] at 286 74
\pinlabel $+1$ [l] at 286 19
\endlabellist
\centering
\includegraphics[scale=1]{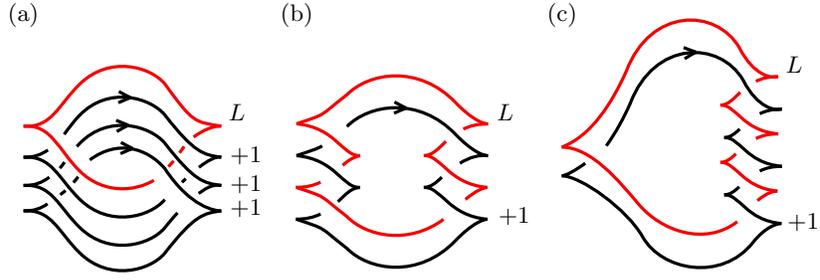}
  \caption{Exceptional rational unknots in projective $3$-space I.}
  \label{figure:rp3-tb3}
\end{figure}

\begin{figure}[h]
\labellist
\small\hair 2pt
\pinlabel $(\rma)$ at 0 185
\pinlabel $(\rmb)$ at 145 185
\pinlabel $n-2$ [r] at 0 96
\pinlabel $n-2$ [l] at 221 96
\pinlabel $L$ [bl] at 69 39
\pinlabel $+1$ [l] at 69 14
\pinlabel $-1$ [l] at 65 56
\pinlabel $-1$ [l] at 65 74
\pinlabel $-1$ [l] at 65 136
\pinlabel $-1$ [l] at 65 167
\pinlabel $L$ [br] at 145 31
\pinlabel $+1$ [r] at 145 23
\pinlabel $-1$ [r] at 155 56
\pinlabel $-1$ [r] at 155 74
\pinlabel $-1$ [r] at 155 136
\pinlabel $-1$ [r] at 155 167
\endlabellist
\centering
\includegraphics[scale=1]{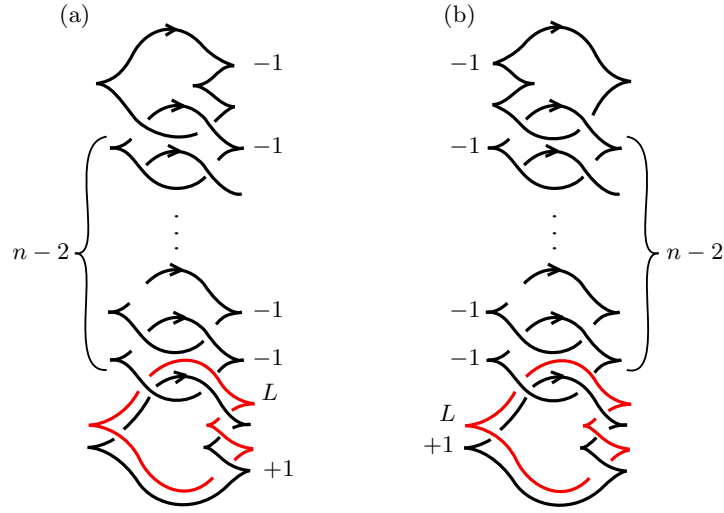}
  \caption{Exceptional rational unknots in projective $3$-space II.}
  \label{figure:rp3-tb5}
\end{figure}

The invariants of these exceptional examples are listed in
Table~\ref{table:rp3}. A sample computation of these invariants is
given in Section~\ref{section:compute}. Since the $d_3$-invariant
differs from $d_3(\xist)=-1/4$, all contact structures given by
these surgery diagrams are overtwisted.

\begin{table}
{\renewcommand{\arraystretch}{1.2}
\begin{tabular}{|c|c|c|c|c|}  \hline
Figure                  & $n$           & $\tb_{\Q}(L)$ &
   $\rot_{\Q}(L)$ & $d_3(\xi)$ \\ \hline
\ref{figure:rp3-tb3}(a) &  -            & 1/2           &
   0              & 1/4        \\ \hline
\ref{figure:rp3-tb3}(b) &  -            & 3/2           &
   0              & 3/4        \\ \hline
\ref{figure:rp3-tb3}(c) &  -            & 3/2           &
   $\pm 1$        & 1/4        \\ \hline
\ref{figure:rp3-tb5}(a) & even $\geq 2$ & $n+1/2$       &
   $\pm(n-1)$     & 3/4        \\ \hline
\ref{figure:rp3-tb5}(b) & odd $\geq 3$  & $n+1/2$       &
   $\pm(n-1)$     & 3/4        \\ \hline
\ref{figure:rp3-tb5}(a) & odd $\geq 3$  & $n+1/2$       &
   $\pm n$        & 1/4        \\ \hline
\ref{figure:rp3-tb5}(b) & even $\geq 2$ & $n+1/2$       &
   $\pm n$        & 1/4        \\ \hline
\end{tabular}
}
\vspace{1.5mm}
\caption{Invariants of the exceptional rational unknots in~$\RP^3$.}
\label{table:rp3}
\end{table}

Except for the example in Figure~\ref{figure:rp3-tb3}(a),
a single contact $(-1)$-surgery along $L$ produces a Stein fillable
contact manifold. In that first example, contact $(-1)$-surgery along
$L$ and two push-offs of $L$, which by the algorithm in
\cite{dgs04} is equivalent to contact $(-1/3)$-surgery along~$L$,
yields $(S^3,\xist)$. So in all cases $L$ is exceptional.

By a similar argument as in the proof of Theorem~\ref{thm:elfr}
we are now going to show that this amounts to a complete list of
the rational unknots in $\RP^3$ up to coarse
equivalence. Given a Legendrian rational unknot $L$ in
$(\RP^3,\xist)$ or an exceptional rational unknot
$L$ in $\RP^3$, we decompose $\RP^3$ into two solid tori
$V_1,V_2$, with $V_1$ a standard neighbourhood of~$L$.
From the standard surgery picture in Figure~\ref{figure:standard-unknot}
we see that the gluing of $V_1$ and $V_2$ is given by
$\mu_2=-\mu_1+2\lambda_1$ and $\lambda_2=\lambda_1$. Suppose
the contact framing of $L$ is $\lambdac=n\mu_1+\lambda_1$
for some $n\in\Z$. Then
\[ \lambdac\bullet\mu_2=
(n\mu_1+\lambda_1)\bullet (-\mu_1+2\lambda_1)=2n+1,\]
hence $\tb_{\Q}(L)=n+1/2$.

In order to compute the slope of the convex torus $\partial V_2$,
we need to express $\lambda_c$ in terms of $\mu_2$ and $\lambda_2$:
\[ \lambdac=n\mu_1+\lambda_1=-n\mu_2+(2n+1)\lambda_2.\]
So the slope is $s_2=-2-1/n$.

For $n\leq -1$ the result of Giroux and Honda quoted
in the proof of Theorem~\ref{thm:elfr} tells us that
there are $|n|$ distinct tight contact structures on~$V_2$.
These are all realised as the complement of a rational
unknot $L_0$ in the tight $(\RP_3,\xist)$ with $\tb_{\Q}(L_0)=n+1/2$.

For $n=0$ the slope $s_2$ is infinite. This can be changed to
$-1$ by a single Dehn twist. So there is a unique tight
contact structure on~$V_2$. For $n\geq 1$, it is easy to see
inductively that the slope $s_2=-2-1/n$ has the continued fraction
expansion $[-3,-2,\ldots ,-2]$, where $-2$ occurs $n-1$
times. So, by Giroux and Honda, there are three tight structures for $n=1$,
and four each for $n\geq 2$. In all cases, this equals the number
of examples in Figures \ref{figure:rp3-tb3} and~\ref{figure:rp3-tb5}.
\end{proof}
\section{The lens spaces $L(p,1)$}
\label{section:Lp1}
The discussion of the preceding section easily generalises
to the lens spaces $L(p,1)$. The following theorem subsumes
Theorems \ref{thm:elfr} and~\ref{thm:rp3}. Part (a) is
essentially the same as~\cite[Theorem~5.5]{baet12}; again, we
state it here merely for completeness and comparison with the
exceptional case.

\begin{thm}
\label{thm:Lp1}
{\rm (a)} Let $L\subset (L(p,1),\xi)$ be a Legendrian rational unknot
in a tight contact structure on $L(p,1)$. Then $\tb_{\Q}(L)=n+1/p$
with $n$ a negative integer, and $\rot_{\Q}(L)$ is of the form
\[ \rot_{\Q}(L)=r_0+\frac{r_1}{p}\]
with
\[ r_0\in\{ n+1,n+3,\ldots,-n-3,-n-1\}\]
and
\[ r_1\in\{ -p+2,-p+4,\ldots ,p-4,p-2\}.\]
Any such pair of invariants $(\tb_{\Q},\rot_{\Q})$ is realised, and it
determines $L$ up to coarse equivalence,
i.e.\ for each $n\leq -1$ we have $|n|\cdot (p-1)$ distinct Legendrian
rational unknots.

{\rm (b)} Up to coarse equivalence, the exceptional rational unknots
in an overtwisted $(L(p,1),\xi)$, $p\in\N$, are classified by their
classical invariants $\tb_{\Q}$ and $\rot_{\Q}$. The possible values
of $\tb_{\Q}$ are $n+1/p$ with $n\in\N_0$. For $n=0$,
there is a single exceptional knot, and it has $\rot_{\Q}=0$.
For $n=1$, there are $p+1$ exceptional knots, with $\rot_{\Q}$ lying
in the range
\[ \bigl\{ -1,-1+\frac{2}{p},-1+\frac{4}{p},\ldots ,-1+\frac{2p}{p}=+1
\bigr\}. \]
For $n\geq 2$, there are $2p$ exceptional knots, with $\rot_{\Q}$
in the range
\[ \bigl\{ \pm \bigl(n-2+\frac{2}{p}\bigr),
\pm \bigl(n-2+\frac{4}{p}\bigr),\ldots,
\pm \bigl(n-2+\frac{2p}{p}\bigr)=\pm n\bigr\}.\]
\end{thm}

\begin{proof}
Legendrian rational unknots in some tight contact structure
on $L(p,1)$ can be found as follows.
Take any tight $L(p,1)$ given by
a surgery diagram as described after Proposition~\ref{prop:d3};
this gives $p-1$ possibilities.
Choose $L$ to be a Legendrian unknot forming a Hopf link with the
surgery curve, with $\tb_0=n$ and $\rot_0$ in the range
\[ \{ n+1,n+3,\ldots ,-n-3,-n-1\}; \]
this gives us $|n|$ choices. With Lemma~\ref{lem:loss}
one easily checks that the invariants of these examples 
are as listed in the theorem.

Examples of exceptional rational unknots whose invariants
have the values stated in the theorem are shown in Figure~\ref{figure:Lp1}.

\begin{figure}[h]
\labellist
\small\hair 2pt
\pinlabel $(\rma)$ at 0 290
\pinlabel $(\rmc)$ at 160 290
\pinlabel $(\rmb)$ at 0 123
\pinlabel $L$ [br] at 4 261
\pinlabel $+1$ [r] at 4 217
\pinlabel $+1$ [r] at 4 225
\pinlabel $+1$ [r] at 4 233
\pinlabel $+1$ [r] at 4 251
\pinlabel $p+1$ [l] at 105 234
\pinlabel $k+1$ [r] at 0 58
\pinlabel $p+1-k$ [l] at 107 58
\pinlabel $L$ [bl] at 76 116
\pinlabel $+1$ [tl] at 72 8
\pinlabel $n-2$ [r] at 158 126
\pinlabel $k$ [r] at 158 233
\pinlabel $p+1-k$ [l] at 250 233
\pinlabel $L$ [br] at 166 40
\pinlabel $+1$ [tl] at 230 8
\pinlabel $-1$ [bl] at 221 279
\pinlabel $-1$ [l] at 242 173
\pinlabel $-1$ [l] at 242 94
\pinlabel $-1$ [l] at 242 72
\endlabellist
\centering
\includegraphics[scale=1]{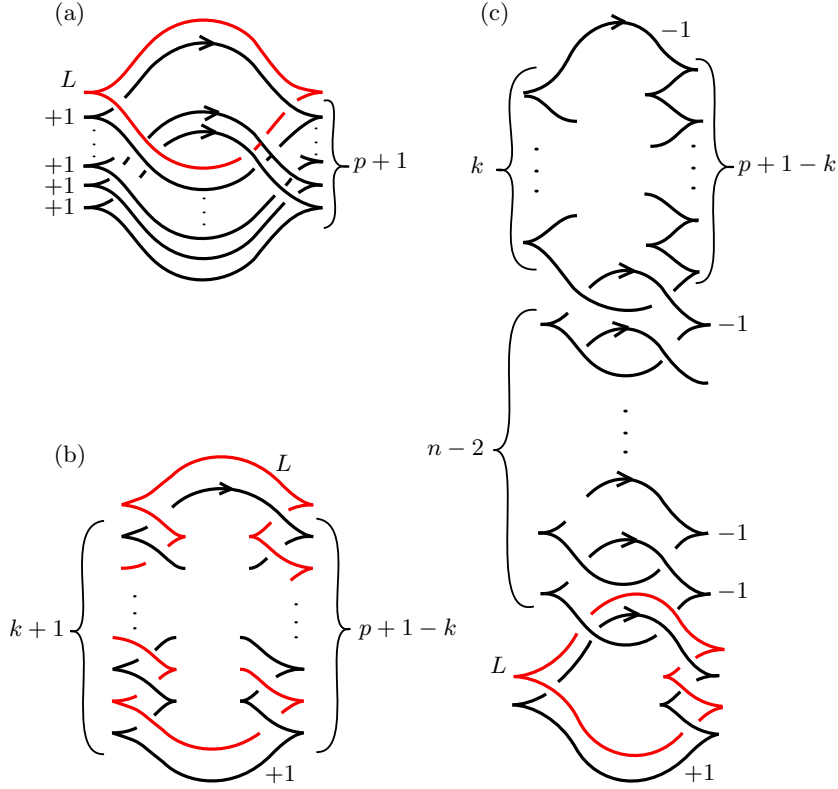}
  \caption{The exceptional rational unknots in $L(p,1)$.}
  \label{figure:Lp1}
\end{figure}

The labels are to be understood as follows. For instance, in
Figure~\ref{figure:Lp1}(b) the surgery knot (and likewise~$L$)
has $k+1$ left-cusps on the left and $p+1-k$ right-cusps
on the right, $k=0,\ldots ,p$. This means that $\tb_0=-(p+1)$
and $\rot_0=p-2k$ (for $L$ oriented clockwise). In Figure~\ref{figure:Lp1}(c)
we take $k$ in the range $1,\ldots ,p$ and $n\geq 2$;
using either orientation for $L$ is going to give us the required
$2p$ examples.
Table~\ref{table:Lp1} summarises the invariants of all
exceptional examples. We defer the computations to
Section~\ref{section:compute}.

\begin{table}
{\renewcommand{\arraystretch}{1.6}
\begin{tabular}{|c|c|c|c|c|c|}  \hline
Figure              & $n$           & $k$          & $\tb_{\Q}(L)$   &
   $\rot_{\Q}(L)$        & $d_3(\xi)$                           \\ \hline
\ref{figure:Lp1}(a) &  -            & -            & $1/p$           &
   0                     & $(3-p)/4$                            \\ \hline
\ref{figure:Lp1}(b) &  -            & $0,\ldots,p$ & $1+\frac{1}{p}$ &
$-1+\frac{2k}{p}$        & $-\frac{(p-2k)^2}{4p}+\frac{3}{4}$   \\ \hline
\ref{figure:Lp1}(c) & even $\geq 2$ & $1,\ldots,p$ & $n+\frac{1}{p}$ &
$\pm(n-2+\frac{2k}{p})$  & $-\frac{(p-2k)^2}{4p}+\frac{3}{4}$   \\ \hline
\ref{figure:Lp1}(c) & odd $\geq 3$  & $1,\ldots,p$ & $n+\frac{1}{p}$ &
$\pm(n+\frac{2}{p}-\frac{2k}{p})$ & $-\frac{(p-2k+2)^2}{4p}+\frac{3}{4}$
                                                        \\ \hline
\end{tabular}
}
\vspace{1.5mm}
\caption{Invariants of the exceptional rational unknots in $L(p,1)$.
\label{table:Lp1}}
\end{table}

In order to illustrate the range of methods available,
we prove overtwistedness of the contact structures on $L(p,1)$
represented in Figure~\ref{figure:Lp1} by a different
argument for each of the three diagrams.

For Figure~\ref{figure:Lp1}(a) we appeal to the
classification of tight contact structures on lens 
spaces~\cite{giro00,hond00}. All these structures are Stein fillable.
Now take $p-2$ additional parallel unknots and perform
contact $(-1)$-surgery along them. This produces the diagram
from Figure~\ref{figure:rp3-tb3}(a), and hence an overtwisted
contact structure on $\RP^3$. If the original surgery diagram
had produced a tight (and hence, in this
particular case, Stein fillable) structure, the resulting structure
on $\RP^3$ would still be Stein fillable, and hence tight.

For Figure~\ref{figure:Lp1}(b), we base our argument on
the rational Bennequin inequality
\[ \tb_{\Q}(L)+|\rot_{\Q}(L)|\leq-\frac{1}{r}\chi(\Sigma);\]
this inequality holds for any rationally null-homologous Legendrian knot
of order $r$ with rational Seifert surface $\Sigma$ in any
\emph{tight} contact $3$-manifold~\cite[Theorem~2.1]{baet12}.
Since the knot $L$ in Figure~\ref{figure:Lp1}(b) violates this
inequality, the manifold given by that surgery diagram must be overtwisted.

In Figure~\ref{figure:Lp1}(c) one may consider a Legendrian
unknot with $\tb=-1$ and $\rot=0$ forming a Hopf link
with the `shark' at the bottom of the picture. As in
\cite[Figure~2]{dgs04} one sees that this Legendrian unknot is
the boundary of an overtwisted disc in the surgered manifold;
the other surgery curves do not intersect this disc.

In each of the examples in Figure~\ref{figure:Lp1},
$L$ is exceptional by the same reasoning
as in the case of $\RP^3$ (in case (a) perform a $(-1/(p+1))$-surgery
along~$L$).

The argument that our list of examples is complete is very similar
to the case of $\RP^3$, and we only list a few of the
necessary modifications. The gluing of $V_1$ and $V_2$ is now
given by $\mu_2=-\mu_1+p\lambda_1$ and $\lambda_2=\lambda_1$.
With $\lambdac=n\mu_1+\lambda_1$, $n\in\Z$, we get $\tb_{\Q}(L)=n+1/p$.
The corresponding slope of $\partial V_2$ is $s_2=-p-1/n$.

For $n\leq -1$, there are $(p-1)\cdot|n|$ distinct contact
structures on~$V_2$. These correspond to the rational unknots
in a tight $L(p,1)$.

For $n=0$ there is again a unique tight contact structure on~$V_2$.
For $n\geq 1$, the slope $-p-1/n$ has the continued fraction
expansion $[-p-1,-2,\ldots ,-2]$, where $-2$ occurs $n-1$ times.
So we have $p+1$ tight structures for $n=1$, and $2p$ each for
$n\geq 2$.
\end{proof}
\section{The lens space $L(5,2)$}
\label{section:L52}
We expect the analogue of Theorem~\ref{thm:Lp1} to hold for
arbitrary lens spaces $L(p,q)$. The number of
Legendrian realisations of the (at most) two rational unknots
in $L(p,q)$ can be computed as before, and one can also
develop some systematics in the surgery diagrams.

Instead of giving this general picture, we concentrate
on one specific example, the lens space $L(5,2)$
and the two topological types $K_1,K_2$ of rational
unknots described in Figure~\ref{figure:K1K2}.
This example serves to illustrate a `stable' pattern
in the surgery diagrams: for sufficiently large values
of $\tb_{\Q}$, there is essentially one general diagram
that covers all cases; for small values of $\tb_{\Q}$
one needs to find some \emph{ad hoc} diagrams. The diagrams
for the `stable' situation generalise in a straightforward manner
to $L(p,q)$.

For $L(5,2)$, the gluing map for the two Heegaard tori is given by
$\mu_2=-2\mu_1+5\lambda_1$ and $\lambda_2=\mu_1-2\lambda_1$.
For the contact framing $\lambdac=n\mu_1+\lambda_1$ of
a Legendrian realisation of $K_1$ one then computes
$\tb_{\Q}=n+2/5$. The corresponding slope of the complementary
solid torus $V_2$ is then equal to $(5n+2)/(2n+1)$, which after
a single Dehn twist becomes
\[ s_2'=-1 -\frac{2n+1}{3n+1}<-1.\]
In Table~\ref{table:L52-numbers1} we list the continued fraction
expansions of this slope and the corresponding number
of tight contact structures on~$V_2$, which gives us the number
of Legendrian realisations of $K_1$ with tight complements.

\begin{table}
\begin{tabular}{|c|c|c|}  \hline
$n$           & c.f.e.\ of $s_2'$          & $\#$ Leg.\ real.    \\ \hline
$\leq -2$     & $[-2,-3,n]$              & $|2n|$            \\ \hline
$-1$          & $[-2,-2]$                & $2$               \\ \hline
$0$           & $-2$                     & $2$               \\ \hline
$1$           & $[-2,-4]$                & $4$               \\ \hline
$\geq 2$      & $[-2,-4,-2,\ldots ,-2]$  & $6$               \\ \hline
\end{tabular}
\vspace{1.5mm}
\caption{Number of Legendrian realisations of $K_1$ in $L(5,2)$.
\label{table:L52-numbers1}}
\end{table}

The cases with $n\leq -1$ correspond to a tight contact structure
on $L(5,2)$ as follows. Realise $L(5,2)$ by contact $(-1)$-surgeries
along a `shark' and a standard $\tb=-1$ Legendrian unknot forming a Hopf link.
A standard Legendrian unknot linked once with the shark gives
a Legendrian realisation of $K_1$ with $\tb_{\Q}=-1+2/5$.
Depending on a choice of orientation, this has $\rot_{\Q}=\pm 2/5$.
By successive stabilisations of this knot, one obtains the
$|2n|$ realisations with $\tb_{\Q}=n+2/5$.

The surgery pictures of the exceptional realisations
of $K_1$ are given in Figures~\ref{figure:L52K1a} and~\ref{figure:L52K1b};
the invariants are listed in Table~\ref{table:L52K1}.
The computations follow the same pattern as in the case of
$L(p,1)$, so we shall not reproduce them here.

\begin{figure}[h]
\labellist
\small\hair 2pt
\pinlabel $(\rma)$ at 0 88
\pinlabel $+1$ [l] at 75 29
\pinlabel $+1$ [r] at 0 36
\pinlabel $+1$ [r] at 0 46
\pinlabel $+1$ [r] at 0 56
\pinlabel $+1$ [br] at 18 69
\pinlabel $+1$ [br] at 77 72
\pinlabel $-1$ [bl] at 134 71
\pinlabel $L$ [tl] at 59 12
\pinlabel $(\rmb)$ at 167 88
\pinlabel $+1$ [tl] at 215 10
\pinlabel $-1$ [bl] at 248 47
\pinlabel $L$ [br] at 181 58
\pinlabel $(\rmc)$ at 296 88
\pinlabel $+1$ [tl] at 354 9
\pinlabel $-1$ [bl] at 382 66
\pinlabel $L$ [br] at  313 58
\endlabellist
\centering
\includegraphics[scale=0.8]{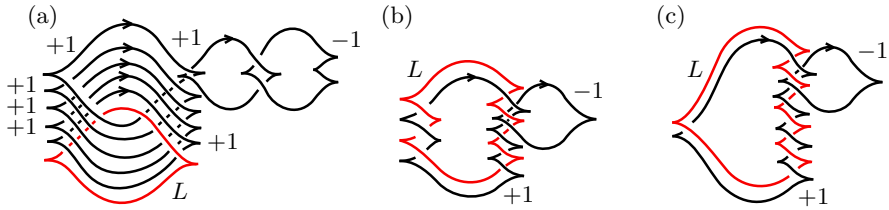}
  \caption{Exceptional rational unknots in $L(5,2)$ isotopic to $K_1$ I.}
  \label{figure:L52K1a}
\end{figure}

\begin{figure}[h]
\labellist
\small\hair 2pt
\pinlabel $(\rma)$ at 0 202
\pinlabel $+1$ [l] at 68 14
\pinlabel $-1$ [l] at 64 56
\pinlabel $-1$ [l] at 64 74
\pinlabel $-1$ [l] at 64 136
\pinlabel $n-2$ [r] at 0 96
\pinlabel $-1$ [r] at 2 168
\pinlabel $-1$ [l] at 91 185
\pinlabel $(\rmb)$ at 124 202
\pinlabel $+1$ [l] at 194 14
\pinlabel $-1$ [l] at 190 56
\pinlabel $-1$ [l] at 190 74
\pinlabel $-1$ [l] at 190 136
\pinlabel $n-2$ [r] at 124 96
\pinlabel $-1$ [r] at 138 170
\pinlabel $-1$ [l] at 222 173
\pinlabel $(\rmc)$ at 248 202
\pinlabel $+1$ [l] at 317 14
\pinlabel $-1$ [l] at 313 56
\pinlabel $-1$ [l] at 313 74
\pinlabel $-1$ [l] at 313 136
\pinlabel $n-2$ [r] at 248 96
\pinlabel $-1$ [r] at 258 183
\pinlabel $-1$ [l] at 344 170
\endlabellist
\centering
\includegraphics[scale=0.8]{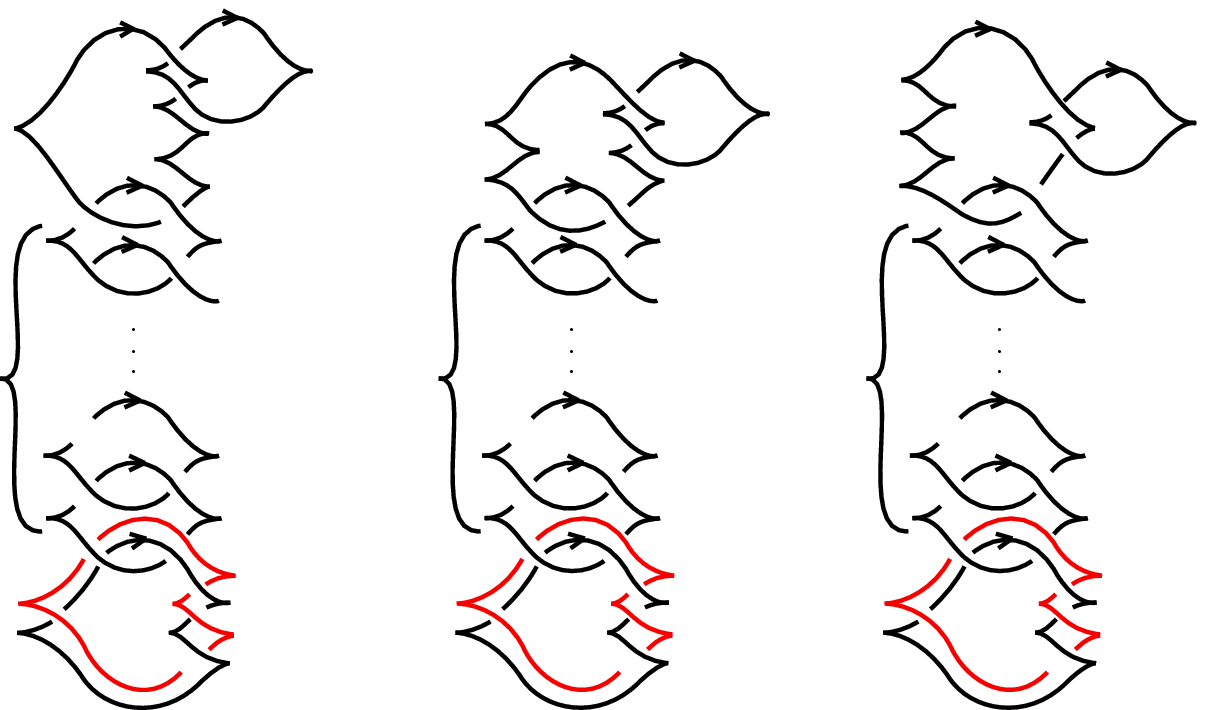}
  \caption{Exceptional rational unknots in $L(5,2)$ isotopic to $K_1$ II.}
  \label{figure:L52K1b}
\end{figure}

\begin{table}
{\renewcommand{\arraystretch}{1.1}
\begin{tabular}{|c|c|c|c|}  \hline
Figure                 & $n$           & $\tb_{\Q}(L)$  &
  $\rot_{\Q}(L)$         \\ \hline
\ref{figure:L52K1a}(a) &  -            & $2/5$          & 
  $\pm 1/5$ \\ \hline
\ref{figure:L52K1a}(b) &  -            & $7/5$          &
  $\pm 2/5$ \\ \hline
\ref{figure:L52K1a}(c) &  -            & $7/5$          &
  $\pm 6/5$ \\ \hline
\ref{figure:L52K1b}(a) & even $\geq 2$ & $n+2/5$        &
  $\pm (n-7/5)$ \\ \hline
\ref{figure:L52K1b}(a) & odd $\geq 3$  & $n+2/5$        &
  $\pm (n+1/5)$ \\ \hline
\ref{figure:L52K1b}(b) &     $\geq 2$  & $n+2/5$        &
  $\pm (n-3/5)$ \\ \hline
\ref{figure:L52K1b}(c) & even $\geq 2$ & $n+2/5$        &
  $\pm (n+1/5)$ \\ \hline
\ref{figure:L52K1b}(c) & odd $\geq 3$  & $n+2/5$        &
  $\pm (n-7/5)$ \\ \hline
\end{tabular}
}
\vspace{1.5mm}
\caption{Invariants of the exceptional realisations of $K_1$ in $L(5,2)$.}
\label{table:L52K1}
\end{table}

The Legendrian realisations of $K_2$ in $L(5,2)$ have $\tb_{\Q}=n+3/5$.
The numbers of different realisations are listed in
Table~\ref{table:L52-numbers2}. Again, the cases with $n\leq -1$
correspond to a tight contact structure on $L(5,2)$, and
they are realised in a similar fashion as the tight cases for~$K_1$.
For the exceptional realisations of $K_2$, see
Figures~\ref{figure:L52K2a} and~\ref{figure:L52K2b}
and Table~\ref{table:L52K2}.

\begin{table}
\begin{tabular}{|c|c|c|}  \hline
$n$           & c.f.e.\ of $s_1'$          & $\#$ Leg.\ real.    \\ \hline
$\leq -2$     & $[-3,-2,n]$              & $|2n|$            \\ \hline
$-1$          & $-2$                     & $2$               \\ \hline
$0$           & $-3$                     & $3$               \\ \hline
$1$           & $[-3,-3]$                & $6$               \\ \hline
$\geq 2$      & $[-3,-3,-2,\ldots ,-2]$  & $8$               \\ \hline
\end{tabular}
\vspace{1.5mm}
\caption{Number of Legendrian realisations of $K_2$ in $L(5,2)$.
\label{table:L52-numbers2}}
\end{table}

\begin{figure}[h]
\labellist
\small\hair 2pt
\pinlabel $(\rma)$ at 57 179
\pinlabel $-1$ [r] at 57 153
\pinlabel $+1$ [l] at 166 133
\pinlabel $+1$ [l] at 166 142
\pinlabel $+1$ [l] at 166 151
\pinlabel $L$ [tl] at 149 116
\pinlabel $(\rmb)$ at 225 179
\pinlabel $+1$ [l] at 301 133
\pinlabel $+1$ [r] at 226 140
\pinlabel $+1$ [r] at 226 151
\pinlabel $+1$ [br] at 242 162
\pinlabel $+1$ [br] at 306 172
\pinlabel $-1$ [l] at 366 166
\pinlabel $L$ [tl] at 284 116
\pinlabel $(\rmc)$ at 0 88
\pinlabel $+1$ [l] at 105 21
\pinlabel $-1$ [bl] at 43 77
\pinlabel $L$ [l] at 105 64
\pinlabel $(\rmd)$ at 148 88
\pinlabel $+1$ [l] at 246 30
\pinlabel $-1$ [bl] at 193 77
\pinlabel $L$ [l] at 246 55
\pinlabel $(\rme)$ at 301 88
\pinlabel $+1$ [l] at 408 22
\pinlabel $-1$ [bl] at 345 77
\pinlabel $L$ [bl] at 381 55
\endlabellist
\centering
\includegraphics[scale=0.78]{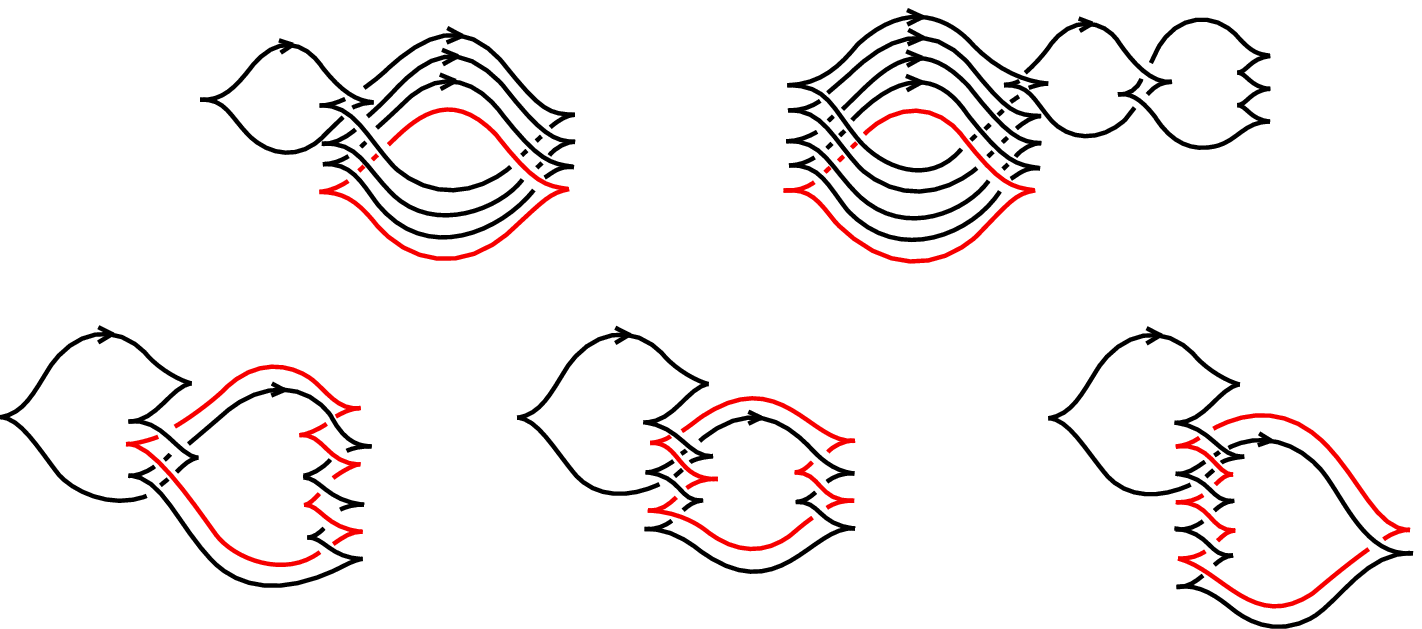}
  \caption{Exceptional rational unknots in $L(5,2)$ isotopic to $K_2$ I.}
  \label{figure:L52K2a}
\end{figure}

\begin{figure}[h]
\labellist
\small\hair 2pt
\pinlabel $(\rma)$ at 0 206
\pinlabel $-1$ [r] at 0 175
\pinlabel $-1$ [l] at 91 166
\pinlabel $-1$ [l] at 94 136
\pinlabel $-1$ [l] at 94 74
\pinlabel $-1$ [l] at 94 56
\pinlabel $+1$ [l] at 97 14
\pinlabel $L$ [l] at 97 39
\pinlabel $n-2$ [r] at 30 95
\pinlabel $(\rmb)$ at 144 206
\pinlabel $-1$ [r] at 145 172
\pinlabel $-1$ [l] at 232 166
\pinlabel $-1$ [l] at 234 136
\pinlabel $-1$ [l] at 234 74
\pinlabel $-1$ [l] at 234 56
\pinlabel $+1$ [l] at 237 14
\pinlabel $L$ [l] at 239 39
\pinlabel $n-2$ [r] at 169 95
\pinlabel $(\rmc)$ at 289 206
\pinlabel $-1$ [r] at 290 181
\pinlabel $-1$ [l] at  377 161
\pinlabel $-1$ [r] at 322 136
\pinlabel $-1$ [r] at 322 74
\pinlabel $-1$ [r] at 322 56
\pinlabel $+1$ [l] at 376 14
\pinlabel $L$ [l] at 378 39
\pinlabel $n-2$ [l] at 389 95
\pinlabel $(\rmd)$ at 427 206
\pinlabel $-1$ [r] at 427 181
\pinlabel $-1$ [l] at 516 162
\pinlabel $-1$ [r] at 461 136
\pinlabel $-1$ [r] at 461 74
\pinlabel $-1$ [r] at 461 56
\pinlabel $+1$ [l] at 516 14
\pinlabel $L$ [l] at 518 39
\pinlabel $n-2$ [l] at 528 95
\endlabellist
\centering
\includegraphics[scale=0.58]{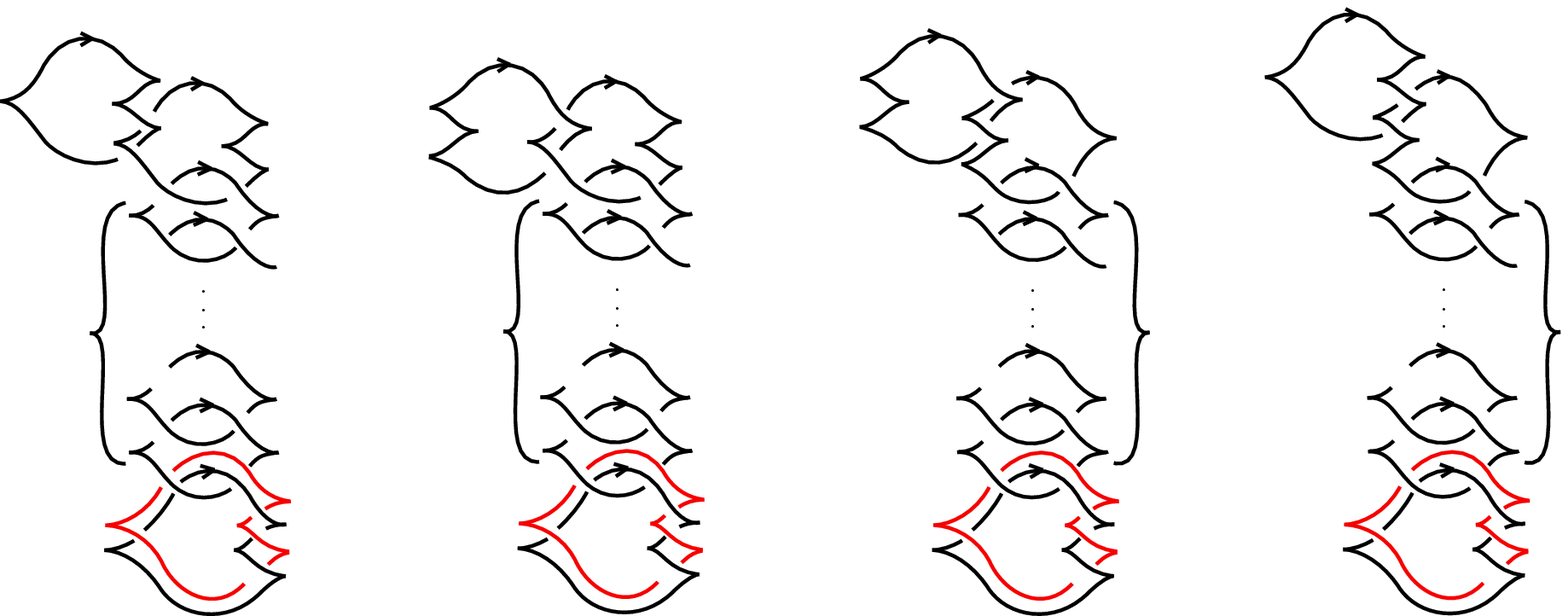}
  \caption{Exceptional rational unknots in $L(5,2)$ isotopic to $K_2$ II.}
  \label{figure:L52K2b}
\end{figure}

\begin{table}
{\renewcommand{\arraystretch}{1.1}
\begin{tabular}{|c|c|c|c|}  \hline
Figure                 & $n$           & $\tb_{\Q}(L)$  &
  $\rot_{\Q}(L)$         \\ \hline
\ref{figure:L52K2a}(a) &  -            & $3/5$          & 
  $0$ \\ \hline
\ref{figure:L52K2a}(b) &  -            & $3/5$          &
  $\pm 2/5$ \\ \hline
\ref{figure:L52K2a}(c) &  -            & $8/5$          &
  $\pm 7/5$ \\ \hline
\ref{figure:L52K2a}(d) &  -            & $8/5$          &
  $\pm 1/5$ \\ \hline
\ref{figure:L52K2a}(e) &  -            & $8/5$          &
  $\pm 1$   \\ \hline
\ref{figure:L52K2b}(a) & even $\geq 2$ & $n+3/5$        &
  $\pm(n-6/5)$   \\ \hline
\ref{figure:L52K2b}(a) & odd $\geq 3$  & $n+3/5$        &
  $\pm(n+2/5)$   \\ \hline
\ref{figure:L52K2b}(b) & even $\geq 2$ & $n+3/5$        &
  $\pm(n-4/5)$   \\ \hline
\ref{figure:L52K2b}(b) & odd $\geq 3$  & $n+3/5$        &
  $\pm n$       \\ \hline
\ref{figure:L52K2b}(c) & even $\geq 2$ & $n+3/5$        &
  $\pm(n+2/5)$   \\ \hline
\ref{figure:L52K2b}(c) & odd $\geq 3$  & $n+3/5$        &
  $\pm(n-6/5)$   \\ \hline
\ref{figure:L52K2b}(d) & even $\geq 2$ & $n+3/5$        &
  $\pm n$       \\ \hline
\ref{figure:L52K2b}(d) & odd $\geq 3$  & $n+3/5$        &
  $\pm(n-4/5)$   \\ \hline
\end{tabular}
}
\vspace{1.5mm}
\caption{Invariants of the exceptional realisations of $K_2$ in $L(5,2)$.}
\label{table:L52K2}
\end{table}

\section{Some computations}
\label{section:compute}
In this section we collect some hints for the
computation of the invariants in various of the examples
described above.
\subsection{Plamenevskaya's examples}
We consider the example in Figure~\ref{figure:olga-lens}(c).
Here the linking matrix $M=M^{(n)}$ is the $((n-1)\times (n-1))$-matrix
\[ M=\left(\begin{array}{ccccccc}
-1 & -1 &    &        &        &    &   \\
-1 & -2 & -1 &        &        &    &   \\
   & -1 & -2 & -1     &        &    &   \\
   &    &    & \ddots &        &    &   \\
   &    &    &        & \ddots &    &   \\
   &    &    &        & -1     & -2 & -1\\
   &    &    &        &        & -1 & -2
\end{array}\right) , \]
where we have numbered the surgery curves $L_1,\ldots ,L_{n-1}$
from bottom to top in the figure.
By successive subtraction of the $i^{\mathrm{th}}$ from the
$(i+1)^{\mathrm{st}}$ row, $i=1,\ldots, n-2$, we obtain
\[ \left(\begin{array}{rrcccrc}
-1 & -1 &    &        &     &   \\
 0 & -1 & -1 &        &     &   \\
   &  0 & -1 & -1     &     &   \\
   &    &    & \ddots &     &   \\
   &    &    &        &  -1 & -1\\
   &    &    &        &   0 & -1
\end{array}\right) , \]
hence $\det M=(-1)^{n-1}$. The first row of $M^{-1}$ is
\[ (-(n-1),n-2,-(n-3),\ldots,(-1)^{n-1}\cdot 1).\]
This information suffices to compute the rotation number of~$L$
(with clockwise orientation):
\[ \rot(L)=1-\left\langle
\left(\begin{array}{c}1\\ 0\\0 \\ \vdots\\ 0\end{array}\right),
M^{-1}
\left(\begin{array}{c}-2\\-1\\0\\ \vdots\\ 0\end{array}\right)
\right\rangle
=1-2(n-1)+n-2=-(n-1).\]

By successive subtraction of the $i^{\mathrm{th}}$ from the
$(i+1)^{\mathrm{st}}$ row, starting from $i=2$, in the
$(n\times n)$-matrix $M_0^{(n)}$ we obtain
\[ \left(\begin{array}{crrccrc}
 0         & -2 & -1 &    &        &    &   \\
-2         & -1 & -1 &    &        &    &   \\
 1         &  0 & -1 & -1 &        &    &   \\
-1         &  0 &  0 & -1 & -1     &    &   \\
\vdots     &    &    &    & \ddots &    &   \\
(-1)^{n-2} &    &    &    &        & -1 & -1\\
(-1)^{n-1} &    &    &    &        &  0 & -1
\end{array}\right) . \]
By expanding the determinant of this matrix along the
last row one shows inductively that
\[ \det M_0^{(n)}=(-1)^{n-1}(n+2).\]
Hence
\[ \tb(L)=-2+\frac{(-1)^{n-1}(n+2)}{(-1)^{n-1}}=n.\]
\subsection{Projective $3$-space}
\label{subsection:rp3}
We consider the example in Figure~\ref{figure:rp3-tb5}(a).
The Kirby moves that transform the surgery link into
a single unknot with topological framing $-2$, and $L$ into 
the rational unknot of Figure~\ref{figure:standard-unknot}, are shown
in Figure~\ref{figure:rp3-kirby}. These moves are analogous to those
in Plamenevskaya's example~\cite[Figure~4]{plam12}; we say more about them
further down, where we use them to compute $\tb_{\Q}(L)$ without appealing
to Lemma~\ref{lem:loss}.

\begin{figure}[h]
\labellist
\small\hair 2pt
\pinlabel (i) at 0 192
\pinlabel $-2$ at 9 169
\pinlabel $L$ [bl] at 40 187
\pinlabel $-1$ [b] at 26 155
\pinlabel $-2$ [tl] at 50 160
\pinlabel $-2$ [tl] at 73 160
\pinlabel $-2$ [t] at 115 160
\pinlabel $-3$ [t] at 136 157
\pinlabel (ii) at 0 131
\pinlabel $L$ [bl] at 40 128
\pinlabel $+1$ [r] at 0 113
\pinlabel $+1$ [r] at 0 104
\pinlabel $+1$ [b] at 26 95
\pinlabel $-2$ [tl] at 50 100
\pinlabel $-2$ [tl] at 73 100
\pinlabel $-2$ [t] at 115 100
\pinlabel $-3$ [t] at 136 97
\pinlabel (iii) at 0 75
\pinlabel $L$ [bl] at 28 73
\pinlabel $+1$ [r] at 1 55
\pinlabel $+1$ [r] at 1 45
\pinlabel $+1$ [t] at 26 34
\pinlabel $-2$ [tl] at 50 41
\pinlabel $-2$ [tl] at 73 41
\pinlabel $-2$ [t] at 115 41
\pinlabel $-3$ [t] at 136 38
\pinlabel (iv) at 0 15
\pinlabel $L$ [r] at 60 16
\pinlabel $-2$ [l] at 101 16
\endlabellist
\centering
\includegraphics[scale=1.2]{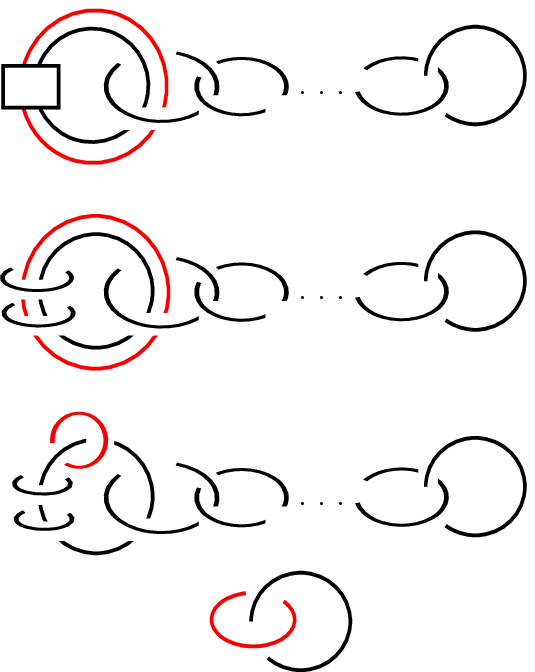}
  \caption{Kirby moves for the example in Figure~\ref{figure:rp3-tb5}(a).}
  \label{figure:rp3-kirby}
\end{figure}

The linking matrix $M=M^{(n)}$ is the $(n\times n)$-matrix
\[ M=\left(\begin{array}{ccccccc}
-1 & -1 &    &        &        &    &   \\
-1 & -2 & -1 &        &        &    &   \\
   & -1 & -2 & -1     &        &    &   \\
   &    &    & \ddots &        &    &   \\
   &    &    &        & \ddots &    &   \\
   &    &    &        & -1     & -2 & -1\\
   &    &    &        &        & -1 & -3
\end{array}\right) , \]
where, as before, we have numbered the surgery curves $L_1,\ldots ,L_n$
from bottom to top. Observe that this $M^{(n)}$ equals the
$M^{(n+1)}$ from the previous example, with a single change in
the very last entry of the matrix. Hence, arguing as before, one
obtains $\det M=(-1)^n\cdot 2$.

The first row of $M^{-1}$ is
\begin{equation}
\label{eqn:row1}
(-(2n-1)/2,(2n-3)/2,-(2n-5)/2,\ldots ,(-1)^n/2);
\end{equation}
the last row is
\begin{equation}
\label{eqn:row2}
((-1)^n/2,(-1)^{n-1}/2,\ldots,1/2,-1/2).
\end{equation}
Hence, with $L$ oriented clockwise,
\begin{eqnarray*}
\rot_{\Q}(L) & = & 1-\left\langle
       \left(\begin{array}{c}1\\ 0\\ \vdots \\ 0\\ 1\end{array}\right),
       M^{-1}
       \left(\begin{array}{c}-2\\-1\\0\\ \vdots\\ 0\end{array}\right)
       \right\rangle \\
 & = & 1-(2n-1-\frac{2n-3}{2}+(-1)^{n-1}-(-1)^{n-1}\cdot\frac{1}{2}) \\
 & = & -n+\frac{1}{2}+(-1)^n\cdot\frac{1}{2} \\
 & = & \begin{cases}
       -(n-1) & \text{for $n$ even,}\\
       -n     & \text{for $n$ odd.}
       \end{cases}
\end{eqnarray*}

By expanding the determinant of $M_0^{(n)}$, transformed as
in the previous example, along the last row, and using the
result from the previous example, one obtains
\[ \det M_0^{(n)}=(-1)^n(2n+5).\]
Hence
\[ \tb_{\Q}(L)=-2+\frac{(-1)^n(2n+5)}{(-1)^n\cdot 2}=n+\frac{1}{2}.\]
Alternatively, $\tb_{\Q}(L)$ can be computed by keeping track
of the framing of $L$ during the Kirby moves in Figure~\ref{figure:rp3-kirby}.
In $S^3$ we have $\tb(L)=-2$, so initially the contact framing of $L$ is
given by $-2\mu_0+\lambda_0$, where $\lambda_0$ is (and remains
throughout the following moves) the longitude
corresponding to the surface framing in $S^3$. This framing
curve can be thought of as a parallel copy of $L$, also going through
the $(-2)$-box.

To get from (i) to~(ii), we make two positive blow-ups, i.e.\ we add two
$(+1)$-framed unknots to the picture (corresponding to taking the
connected sum with two copies of $\CP^2$), and then slide $L$
and the $(-1)$-framed knot over them to undo the $(-2)$-linking.
This adds two positive twists to the framing of $L$, so the framing is
now~$\lambda_0$.

To get from (ii) to (iii), we slide $L$ over the parallel
$(+1)$-framed unknot. This adds $+1$ to the framing of~$L$.
To get from (iii) to (iv), we first blow down the two
$(+1)$-framed unknots not linked with $L$. This has no effect on $L$,
but changes the third $(+1)$-framed unknot into a $(-1)$-framed one.
Now we blow down the chain of unknots, starting with the
$(-1)$-framed one. At each step, the adjacent $(-2)$-framed
unknot gets framing $-1$, and the framing of $L$
increases by~$1$. Since we have to blow down a total
of $n-1$ $(-1)$-framed unknots to obtain (iv), the framing of $L$
finally becomes $n\mu_0+\lambda_0$.

Now recall equation~(\ref{eqn:tb}) from the proof of
Lemma~\ref{lem:loss} and the argument preceding it.
The unique $a_0\in\Z$ such that $a_0\mu_0+2\lambda_0$ is
nullhomologous in the surgered manifold given by
Figure~\ref{figure:rp3-kirby}(iv) is $a_0=-1$. Hence
\[ 2\cdot\tb_{\Q}(L)=(n\mu_0+\lambda_0)\bullet
(-\mu_0+2\lambda_0)=2n+1,\]
giving us the same result for $\tb_{\Q}(L)$ as before.

Here is how to compute $d_3(\xi)$ for this example. The surgery diagram
is equivalent to $n-1$ unlinked $(-1)$-framed unknots and a further
unlinked $(-2)$-framed unknot. So the signature of the
$4$-dimensional filling $X$ is $-n$, its Euler characteristic
is $n+1$. In order to compute $c^2$, we follow the
algorithm described in~\cite{dgs04}. The Poincar\'e
dual $\mathrm{PD}(c)\in H_2(X,\partial X)$ --- in terms of the obvious
generators of $H_2(X,\partial X)$, the meridional discs to
the surgery curves --- is given by the vector $(1,0,\ldots,0,1)$
of rotation numbers. The homomorphism $H_2(X)\rightarrow
H_2(X,\partial X)$ induced by inclusion is described, again in terms of
the obvious bases, by the linking matrix~$M$, so
the class $C\in H_2(X)$ that maps
to $\mathrm{PD}(c)$ can be thought of as a row vector
with $MC^t= (1,0,\ldots,0,1)^t$. So this vector $C$
is given by the sum of the
vectors in (\ref{eqn:row1}) and~(\ref{eqn:row2}). Then
\begin{eqnarray*}
c^2 & = & C^2\; =\; CMC^t\\
 & = & C\cdot(1,0,\ldots,0,1)^t\\
 & = & -\frac{2n-1}{2}+\frac{(-1)^n}{2}
       +\frac{(-1)^n}{2}-\frac{1}{2}\\
 & = & -n+(-1)^n.
\end{eqnarray*}
Then with equation~(\ref{eqn:d3}), observing that $q=1$ in this example,
we obtain
\[ d_3(\xi)=\begin{cases}
            3/4 & \text{for $n$ even},\\
            1/4 & \text{for $n$ odd.}
\end{cases} \]

\subsection{The lens spaces $L(p,1)$.}
We start with the example in Figure~\ref{figure:Lp1}(a).
The Kirby moves in Figure~\ref{figure:Lp1-kirby}
show that $L$ is the rational unknot in $L(p,1)$.

\begin{figure}[h]
\labellist
\small\hair 2pt
\pinlabel $L$ [r] at 2 137
\pinlabel $-1$ at 42 137
\pinlabel $0$ [tl] at 68 128
\pinlabel $0$ [tl] at 79 119
\pinlabel $L$ [r] at 0 78
\pinlabel $+1$ [b] at 61 86
\pinlabel $p+1$ [t] at 63 56
\pinlabel $+1$ [r] at 51 77
\pinlabel $+1$ [r] at 73 77
\pinlabel $L$ [r] at 9 18
\pinlabel $-p$ [l] at 77 20
\endlabellist
\centering
\includegraphics[scale=1]{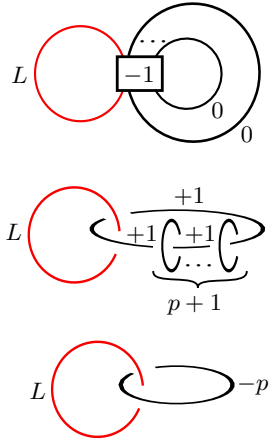}
  \caption{Kirby moves for the example in Figure~\ref{figure:Lp1}(a).}
  \label{figure:Lp1-kirby}
\end{figure}

The linking matrix is the $((p+1)\times (p+1))$-matrix
$M$ with zeros on the diagonal and all other entries equal to~$1$.
It is a simple exercise to show that $\det M=-p$.
Correspondingly, $\det M_0=-(p+1)$. It follows that
\[ \tb_{\Q}(L)=-1+\frac{p+1}{p}=\frac{1}{p}.\]
Since $\rot_0,\rot_1,\ldots ,\rot_{p+1}=0$, we have $\rot_{\Q}(L)=0$.

\vspace{1mm}

Next we consider the example in Figure~\ref{figure:Lp1}(b).
Here the topological Kirby diagram shows
directly that $L$ is the rational unknot in $L(p,1)$.
The linking matrix $M$ is the $(1\times 1)$-matrix $(-p)$,
the matrix $M_0$ is
\[ \left(\begin{array}{cc}
0      & -(p+1)\\
-(p+1) & -p\end{array}\right).\]
Hence
\[ \tb_{\Q}(L)=-(p+1)+\frac{(p+1)^2}{p}=1+\frac{1}{p}.\]
The rotation numbers $\rot_i$, $i=0,1$, equal $p-2k$
(with $L$ oriented clockwise). Hence
\[ \rot_{\Q}(L)=p-2k-(p-2k)\cdot\bigl(-\frac{1}{p}\bigr)\cdot \bigl(-(p+1)
\bigr)=-1+\frac{2k}{p}.\]
So for $k$ in the range $0,\ldots,p$ we get $p+1$ different
Legendrian realisations.

\vspace{1mm}

Finally, we come to the example in Figure~\ref{figure:Lp1}(c).
Here the computations are minor modifications
of those for the example we discussed in the case of $\RP^3$.
The Kirby moves for showing that $L$ is the rational unknot are
as in Figure~\ref{figure:rp3-kirby}. The linking matrix $M$
differs from the one in that previous case by the
substitution of $-(p+1)$ for~$-3$. Thus, one finds
$\det M=(-1)^n\cdot p$ and $\det M_0=(-1)^n\bigl( p(n+2)+1\bigr)$,
which yields
\[ \tb_{\Q}(L)=n+\frac{1}{p}.\]
The first row of $M^{-1}$ is now
\[ \bigl( -\frac{(n-1)p+1}{p},\frac{(n-2)p+1}{p},\ldots,
(-1)^{n-1}\frac{p+1}{p},(-1)^n\frac{1}{p}\bigr);\]
the last row is
\[ ((-1)^n/p,(-1)^{n-1}/p,\ldots,1/p,-1/p).\]
With the rotation number of the surgery curve at the top
of Figure~\ref{figure:Lp1}(c) being $\rot_n=p-2k+1$, one computes
with Lemma~\ref{lem:loss} that $\rot_{\Q}(L)$, with either
orientation of $L$, takes for $k=1,\ldots ,p$ the values
claimed in Theorem~\ref{thm:Lp1}.

\vspace{1mm}

We close with some comments about the computation of the
$d_3$-invariant. For the example in Figure~\ref{figure:Lp1}(a),
the only term in formula (\ref{eqn:d3}) that is not entirely
obvious is the signature. Topologically, the surgery diagram
consists of $p+1$ $0$-framed unknots with a common
$(-1)$-linking. By making a $(+1)$-blow up and sliding the
corresponding $(+1)$-framed unknot over this link, we obtain
$p+1$ unlinked $(+1)$-framed unknots, all of which are linked
once with the extra $(+1)$-framed unknot. By sliding the
$p+1$ unknots off the extra one, we obtain an unlink
consisting of a single $(-p)$-framed unknot and $p+1$ $(+1)$-framed
unknots. This describes a filling of signature~$p$. Since we had to
add a $(+1)$-framed unknot to arrive at this picture, we have $\sigma=p-1$.

The computation of $d_3$ for the example in
Figure~\ref{figure:Lp1}(b) presents no difficulty.

For the example in Figure~\ref{figure:Lp1}(c), one sees
$\sigma=-n$ by an argument similar to that for~(a). Since the
surgery diagram corresponds to adding $n$ $2$-handles,
we have $\chi=1+n$. The vector of rotation numbers
is given by $(1,0,\ldots,0,p-2k+1)$, i.e.\ we need to solve
the equation
\[ MC^t=(1,0,\ldots,0,p-2k+1)^t\]
over~$\Q$. This is achieved by
\[ C=\begin{cases}
\frac{1}{p}\bigl(-(2k+(n-2)p),+(2k+(n-3)p),\ldots,
-2k,+(2k-p)\bigr) & \text{for $n$ even},\\
\frac{1}{p}\bigl(+(2k-np-2),\ldots,-(2k-2p-2),+(2k-p-2)\bigr) &
\text{for $n$ odd}.
\end{cases} \]
Then one computes as in Section~\ref{subsection:rp3}.
\begin{ack}
A major part of the work on this project was done during an inspiring
``Research in Pairs'' stay at the Mathematisches Forschungsinstitut
Oberwolfach in September 2012. We thank the Forschungsinstitut for
its support, and its efficient and friendly staff for creating
an exceptional research environment.
\end{ack}

\end{document}